\documentclass{article}
\usepackage{ulem}

\usepackage{amssymb,amsmath,graphicx,amsthm,mathrsfs,bm}
\usepackage[textsize=footnotesize,color=DarkGreen!40]{todonotes}

\usepackage[colorlinks=true]{hyperref}
\usepackage{pdfsync,color}
\allowdisplaybreaks[4]
\numberwithin{equation}{section}
\newtheorem{theorem}{Theorem}[section]

\newtheorem{lemma}[theorem]{Lemma}

\theoremstyle{definition}

\theoremstyle{definition}
\newtheorem{remark}[theorem]{Remark}
\numberwithin{equation}{section}

\newcommand{\lc}
{\mathrel{\raise2pt\hbox{${\mathop<\limits_{\raise1pt\hbox
{\mbox{$\sim$}}}}$}}}

\newcommand{\gc}
{\mathrel{\raise2pt\hbox{${\mathop>\limits_{\raise1pt\hbox{\mbox{$\sim$}}}}$}}}

\newcommand{\ec}
{\mathrel{\raise2pt\hbox{${\mathop=\limits_{\raise1pt\hbox{\mbox{$\sim$}}}}$}}}

\def\bb{\begin{equation}} \def\ee{\end{equation}}

\def\beqn{\begin{eqnarray}}  \def\eqn{\end{eqnarray}}

\def\beqnx{\begin{eqnarray*}} \def\eqnx{\end{eqnarray*}}

\def\bn{\begin{enumerate}} \def\en{\end{enumerate}}

\def\bd{\begin{description}} \def\ed{\end{description}}

\def\label{\label}

\renewcommand{\leq}{\leqslant}
\renewcommand{\geq}{\geqslant}

\title{Quantitative unique continuation and observability on an equidistributed set for the diffusion equation in $\mathbb R^N$}

\author{Yueliang Duan \thanks{School of
Mathematics and Statistics,
Wuhan University, Wuhan 430072, China;
e-mail: ylduan@whu.edu.cn.},\quad
Huaiqiang Yu \thanks{School of Mathematics,
Tianjin University, Tianjin 300354, China;
e-mail: huaiqiangyu@tju.edu.cn.},
\quad Can Zhang
\thanks{School of Mathematics and Statistics; Computational
Science Hubei Key Laboratory,
Wuhan University, Wuhan 430072, China;
e-mail: zhangcansx@163.com}}

\begin{document}
\maketitle
\begin{abstract}
In this paper, we obtain a quantitative estimate of unique continuation
and an observability inequality from an equidistributed set for solutions of the diffusion equation in the whole space $\mathbb R^N$.  This kind of observability indicates that the total energy of solutions can be controlled by the energy localized in a measurable subset, which is equidistributed over the whole space. The proof of our results is based on an interesting reduction method \cite{LO,Lin90},  as well as the propagation of smallness for the gradient of solutions to elliptic equations \cite{Logunov-Malinnikova}.
\end{abstract}
\vskip 8pt

 \noindent\textbf{Keywords.} Unique continuation, Observability, Equidistributed set,  Diffusion equation.
    \medskip

       \noindent \textbf{2010 AMS Subject Classifications.} 35B60, 35K10, 93C20.
\vskip 10pt

\tableofcontents
\section{Introduction}\label{sec_intro}

Let $N$ be a positive integer, $\mathbb{N}^+:=\{1,2,\ldots\}$ and $\mathbb{R}^+:=(0,+\infty)$. Consider the linear diffusion equation over $\mathbb{R}^N$:
\begin{equation}\label{yu-6-24-1}
\left\{
\begin{array}{lll}
\partial_t u(x,t)-\sum\limits_{j,k=1}^N\partial_j\left(a_{jk}(x)\partial_ku(x,t)\right)=0&\mbox{in}\;\;\mathbb R^N\times\mathbb R^+,\\
		u(x,0)=u_0(x) &\mbox{in}\;\;\mathbb R^N,
\end{array}\right.
\end{equation}
where $u_0\in L^2(\mathbb R^N)$, $a_{jk}(x)=a_{kj}(x)$ for all $j,k=1,\dots,N$, and all $x\in\mathbb R^N$.  We assume the matrix-valued function $A(\cdot)=(a_{jk}(\cdot))_{j,k=1}^N$ is Lipschitz continuous and satisfies the uniform ellipticity condition, i.e.,
	  there is a constant $\lambda\geq1$ such that
\begin{equation}\label{yu-11-28-2}
	|a_{jk}(x)-a_{jk}(y)|\leq \lambda|x-y|\quad\text{and}\quad
    \lambda^{-1}|\xi|^2\leq \sum_{j,k=1}^N a_{jk}(x)\xi_j\xi_k \leq \lambda|\xi|^2
\end{equation}
for all $j,k=1,\dots,N$, $x, y\in\mathbb{R}^N$, and all $\xi=(\xi_1,\xi_2\cdots,\xi_N)\in\mathbb{R}^N$.

According to \cite[Theorem 10.9]{Brezis},   \eqref{yu-6-24-1} has a unique solution
$u\in L^2(\mathbb{R}^+; H^1(\mathbb{R}^N))\cap C([0,+\infty); L^2(\mathbb{R}^N))$.
Furthermore, if $u_0\in H^1(\mathbb R^N),$ then $u\in C([0,+\infty); H^1(\mathbb{R}^N))$
and $\partial_t u\in L^2(\mathbb{R}^+; L^2(\mathbb{R}^N))$.

Throughout the paper, $B_R(x_0)$ stands for an open ball in $\mathbb R^N$ with the center $x_0$ and of radius $R>0$, while
$B_R(x_0,0)$ stands for an open ball in $\mathbb R^{N+1}$ with the center $(x_0,0)\in\mathbb R^{N}\times\mathbb R$ and of radius $R>0$.
We denote the usual inner product and norm in $L^2(B_R(x_0))$ by $\langle\cdot,\cdot\rangle_{L^2(B_R(x_0))}$ and $\|\cdot\|_{L^2(B_R(x_0))}$, respectively. {Define the cube with center $x_0 =(x_{0,1},x_{0,2},\cdots,x_{0,N})$ and side length $R>0$ as $Q_R(x_0):=\{v=(v_1,v_2,\cdots,v_N):|v_i-x_{0,i}|\leq R,\;\forall i=1,2,\ldots,N\}$.  Denote by $\mathrm{int}(Q_{R}(x))$ the interior of $Q_{R}(x)$, and by $\partial B_R(x_0)$ the boundary of $B_R(x_0)$.
Write $\bar z$ for the conjugate of a complex number $z\in\mathbb C$, and $|\omega|$
for the Lebesgue measure for a measurable set $\omega\subset\mathbb R^N$.

\subsection{Main results}
The main results of this paper concerning the quantitative estimate of unique continuation   and the observability inequality  for solutions of \eqref{yu-6-24-1} are stated as follows.
		
\begin{theorem}\label{yu-theorem-7-10-6}
Let $T>0$, $\rho>0$ and $0<r<+\infty$. Let  $\{x_i\}_{i\in\mathbb{N}^+}\subset\mathbb R^N$ be arbitrarily fixed  so that
\begin{equation*}
   \mathbb{R}^{N}=\bigcup_{i\in\mathbb{N}^+}Q_{r}(x_{i})
\quad \text{with}\quad \mathrm{int}(Q_{r}(x_{i}))\bigcap \mathrm{int}(Q_{r}(x_{j}))=\emptyset\quad \text{for each}\quad i\neq j\in\mathbb N^+.
\end{equation*}
   Let $\{\omega_i\}_{i\in\mathbb{N}^+}$ be some N-dimensional Lebesgue measurable sets in $\mathbb{R}^N$ so that
$$\omega_{i} \subset B_{\frac r2}(x_{i})\quad\text{and}\quad \frac{|\omega_{i}|}{|B_{r}(x_i)|}\geq\rho \quad\text{for each}\quad i\in\mathbb N^+. $$
 Then there exist constants $C=C(r,N,\lambda,\rho)>0$ and  $\sigma=\sigma(r,N,\lambda,\rho)\in(0,1)$ such that for any $u_{0}\in L^{2}(\mathbb R^N)$, the corresponding solution $u$ of \eqref{yu-6-24-1} satisfies
\begin{equation*}\label{yu-7-10-2}
	\int_{\mathbb R^N}|u(x,T)|^2dx\leq Ce^{C(T+\frac{1}{T})}\left(\int_{\omega}|u(x,T)|^2dx\right)^{\sigma}\left(\int_{\mathbb R^N}|u_{0}(x)|^2dx\right)^{1-\sigma},
	\end{equation*}
    where $\omega:=\cup_{i\in\mathbb{N}^+} \omega_i$.
\end{theorem}

\medskip

\begin{theorem}\label{jiudu4}
Let $E$ be a subset of positive Lebesgue measure in $(0,T)$. Under the same assumptions in Theorem \ref{yu-theorem-7-10-6}, then there exist constants $C=C(r,N,\lambda,\rho)>0$ and  $\tilde{C}=\tilde{C}(E, r,N,\lambda,\rho)>0$  such that for any $u_{0}\in L^{2}(\mathbb R^N)$, the corresponding  solution $u$ of \eqref{yu-6-24-1} satisfies
	\begin{equation*}
	\int_{\mathbb R^N}|u(x,T)|^2dx\leq e^{\tilde{C}+CT}\int_{\omega\times E}|u(x,t)|^2dxdt.
\end{equation*}
\end{theorem}

\medskip

Several remarks are given below.
\begin{remark}
When $E=(0,T)$,  the constant $\widetilde{C}(E, r, N,\lambda,\rho)$ in
Theorem \ref{jiudu4}  can take the form $\widetilde{C}( r,N,\lambda,\rho)/T$.
\end{remark}

\begin{remark}
When $A(\cdot)\equiv I$ (the $N\times N$ order identity matrix) in the equation \eqref{yu-6-24-1},  the above two theorems have been already established in \cite{WangZhangZhang} for solutions of the heat equation in $\mathbb R^N$, when the observation set is thick. The proofs there are based on quantitative estimates from measurable sets for real analytic functions, and uncertainty principle. To the best of our knowledge, however, this approach cannot be extended to the case of variable coefficients.
In the present paper, we built a new approach to obtain the
observability inequality for the diffusion equation \eqref{yu-6-24-1}.
\end{remark}

\begin{remark}\label{remark409}
Recall that  a subset $\mathcal{O}\subset\mathbb R^N$ is called thick with parameters $L>0$ and $\gamma>0$ means that in each cube $Q_L\subset\mathbb R^N$ with the length $L$, the $N$-dimensional Lebesgue measure of $\mathcal{O}\cap Q_L$ is bigger than or equals to $\gamma L^N$.
As the thick set $\mathcal{O}$ is almost dense in the whole space, one may imagine that such kind of thick set has a close
 relationship with the equidistributed set $\omega$ as defined in
Theorem \ref{yu-theorem-7-10-6}.
Indeed,  an equidistributed set $\omega$ is clearly a thick set.
On the other hand,  if $\mathcal{O}$ is a
thick set, then one could construct an equidistributed
set $\omega$ to be such that $\omega \subset \mathcal{O}$.
\end{remark}

\begin{remark}\label{re522}
By \cite[Lemma 2.5]{WangZhangZhang}, one can obtain that the equidistributed observation set is also a necessary condition such that the observability inequality holds  for the diffusion equation \eqref{yu-6-24-1}.
\end{remark}

\begin{remark}
It is also worth mentioning that the authors in \cite{WangZhangZhang, EV17} have proved that the observability inequality for the heat equation holds if and only if the observation set is thick.  Combined with Remarks \ref{remark409} and \ref{re522} above, one can finally conclude that the observability inequality for the diffusion equation \eqref{yu-6-24-1} holds if and only if the observation set is thick or equidistributed in $\mathbb R^N$.
\end{remark}

\subsection{Motivation and novelty}
The observability inequality for linear parabolic equations on bounded domains are previously studied in a large number of publications. When $E$ is the whole time interval and the observation region $\omega$ is a nonempty open subset, we refer the readers to \cite{FI,fz} and a vast number of references therein for the observability inequality for parabolic equations. Their approach is mainly based on the Carleman inequality method.
When $E$ is only a subset of positive Lebesgue measure in the time interval
and the observation region $\omega$ is a non-empty open subset, we refer the readers to \cite{Apraiz-Escauriaza-Wang-Zhang,Phung-Wang-2013, Phung-Wang-Zhang,wang-zhang1} for the  observability inequality for parabolic equations.  More generally,
when the observation subdomain is a measurable subset of positive measure in the space and time variables, we refer the readers to \cite{Escauriaza-Montaner-Zhang,Escauriaza-Montaner-Zhang2,Apraiz-Escauriaza22} for the observability inequality for analytic parabolic equations. The proofs of those inequalities are mainly based on the parabolic frequency function method, and the propagation of smallness estimate for real analytic functions.
Note that the solution to \eqref{yu-6-24-1} is obviously not spatial real-analyticity any more, and thus we cannot derive the observability inequality from measurable sets by the real analytic technique.

The studies on the observability inequality for parabolic equations on unbounded domains are rather few in last decades. The observability inequality may not be true when the heat equation is evolving in the whole space and the observation subdomain is only a bounded and open subset (see, e.g., \cite{MZ,MZb,M05a}). We would like to mention the work \cite{CMZ} for sufficient conditions so that the observability inequalities hold true for heat equations in unbounded domains. It showed that, for some parabolic equations in an unbounded domain
$\Omega \subset \mathbb{R}^N$,  the observability inequality holds when  observations are made over an open subset $\omega \subset\Omega$, with $\Omega \backslash \omega$ bounded. The proofs are based on Carleman estimates.  For other similar results, we refer the reader  to \cite{B,CMV,Gde,RM,Z16}.

Recently, \cite{WangZhangZhang} and \cite{EV17} independently obtained the observability inequality for the pure heat equation on the whole space, where the observation is a thick set (see Remark \ref{remark409}). This could be extended to
the time-independent parabolic equation associated to the Schr\"odinger operator with analytic coefficients (see \cite{EV19,leb,nttv}).
The methods utilized in these papers are all based on the spectral inequality (or uncertainty principle). Unfortunately, they are not valid any more for the case that the coefficients
in parabolic equations are non-analytic.

A  H$\ddot{\mathrm{o}}$lder type propagation of smallness for solutions, as well as
for their gradients, on a possibly lower dimensional Hausdorff subset to general second order elliptic equations has been established in \cite{Logunov-Malinnikova}. The approach in
\cite{Logunov-Malinnikova} does not base on the analyticity of solutions.
This result stimulates our motivation and opens
a possible way to establish the observability inequality from measurable sets for solutions of general parabolic equations, getting rid of the above-mentioned technique of analyticity.

Recently, quantitative estimates of unique continuation for solutions of second order parabolic equations, such as the doubling property, as well as the two-ball and one-cylinder inequality, have been well understood (see, e.g., \cite{Apraiz-Escauriaza-Wang-Zhang},  \cite{Canuto-Rosset-Vessella}, \cite{Duan-Wang-Zhang}, \cite{Escauriaza-Fernandez-Vessella-2006}, \cite{LO}, \cite{Lin90}, \cite{Phung-Wang-2010}, \cite{Phung-Wang-2013}).

The main effort of this paper is to prove a locally quantitative estimate of strong unique continuation  for solutions to the diffusion equation,  by utilizing the propagation of smallness estimate for gradients of solutions to elliptic equations.
Combing with the above local result and the geometry of the equidistributed observation, we obtain a globally quantitative estimate at one time point for solutions of the diffusion equation. We finally apply the telescoping series method to prove the desired observability inequality.

\subsection{Plan of this paper}
The rest of this paper is organized as follows. In Section \ref{sc2}, we give several auxiliary lemmas.  Sections \ref{kaodu3} and \ref{finalproof}  prove Theorem \ref{yu-theorem-7-10-6} and Theorem \ref{jiudu4},  respectively.

\section{Some auxiliary lemmas}\label{sc2}

In this section, we first show two exponential estimates (see Lemmas \ref{yu-lemma-6-10-1} and \ref{yu-lemma-6-18-1} below)  in Section \ref{kaodu1}, and then give a quantitative estimate of Cauchy uniqueness (see Lemma \ref{yu-proposition-7-1-1} below) in Section \ref{yu-section-7-26-3}.


\subsection{Local energy estimates}\label{kaodu1}

Let $T>0$ and $\tau\in(0,T)$ be arbitrarily fixed.
Let $\eta\in C^\infty(\mathbb{R}^+;[0,1])$ be a cutoff function satisfying
\begin{equation}\label{yu-6-6-6}
\begin{cases}
	\eta= 1 &\mbox{in}\;\;[0,\tau],\\
	\eta=0 &\mbox{in}\;\; [T,+\infty),\\
	|\eta_t|\leq \frac{C_{1}}{T-\tau}&\mbox{in}\;\;(\tau,T),
\end{cases}
\end{equation}
where the positive  constant $C_{1}$ is independent  of $\tau$ and $T$.

Let $R>0$ and $x_0\in\mathbb R^N$.  Let $v$ be the solution to
    \begin{equation}\label{yu-11-29-4}
\begin{cases}
    v_t-\mbox{div}(A(x)\nabla v)=0&\mbox{in}\;\;B_{R}(x_0)\times\mathbb{R}^+,\\
    v=\eta u&\mbox{on}\;\;\partial B_R(x_0)\times\mathbb{R}^+,\\
    v(\cdot,0)=0&\mbox{in}\;\; B_R(x_0),
\end{cases}
\end{equation}
where $u$ satisfies  \eqref{yu-6-24-1} with the initial value $u_0\in H^1(\mathbb{R}^N)$ and $\eta$ verifies \eqref{yu-6-6-6}.

Then, we have the following  exponential decay result.
\begin{lemma}\label{yu-lemma-6-10-1}
There exist positive constants $\hat{C}_{1}=\hat{C}_{1}(\lambda)$ and $\hat{C}_{2}=\hat{C}_{2}(\lambda,N)$ such that
\begin{equation}\label{yu-6-18-1}
 	\|v(\cdot,t)\|_{H^1(B_R(x_0))}\leq \hat{C}_{1}\left(1+\frac{1}{T}\right)^{\frac{1}{2}}
 e^{\frac{\hat{C}_{1}T}{T-\tau}-\hat{C}_{2}R^{-2}(t-T)^+}F(R)\quad\text{for all}\;\; t\in\mathbb{R}^+,
 \end{equation}
where $(t-T)^+:=\max\{0,t-T\}$ and
	$F(R):=\sup\limits_{s\in[0,T]}\|u(\cdot,s)\|_{H^1(B_R(x_0))}$.
\end{lemma}

\begin{proof}
We proceed the proof into two steps as follows.

\par
\vskip 5pt
    \textit{Step 1. To prove \eqref{yu-6-18-1} when $t\in[0,T]$.}

Setting $w=v-\eta u$ in $B_R(x_0)\times\mathbb R^+$, we see that  $w$ verifies
\begin{equation}\label{yu-6-6-7}
\begin{cases}
	w_s-\mbox{div}(A(x)\nabla w)=-\eta_su
	&\mbox{in}\;\;B_R(x_0)\times\mathbb R^+,\\
	w=0&\mbox{on}\;\;\partial B_R(x_0)\times\mathbb R^+,\\
	w(\cdot,0)=-u_{0}&\mbox{in}\;\;B_R(x_0).
\end{cases}
\end{equation}
	We next prove that for each $t\in[0,T]$,
\begin{equation}\label{yu-6-7-6}
	\|w(\cdot,t)\|_{L^2(B_R(x_0))}\leq e^{\frac{C_{1}T}{T-\tau}}
	\sup_{s\in[0,T]}\|u(\cdot,s)\|_{L^2(B_R(x_0))}.
\end{equation}

Indeed, multiplying  first (\ref{yu-6-6-7}) by $w$ and then integrating by parts over $B_R(x_0)\times(0,t)$ lead to
\begin{eqnarray*}\label{yu-6-7-1}
	&\;&\| w(\cdot,t)\|_{L^2(B_R(x_0))}^2+2\int_0^t\int_{B_R(x_0)}A\nabla
w\cdot\nabla w\,dxds\nonumber\\
	&\leq&\| u_{0}\|_{L^2(B_R(x_0))}^2+\frac{C_{1}t}{T-\tau}\sup_{s\in[0,T]}\|u(\cdot,s)\|^2_{L^2(B_R(x_0))}+\frac{C_{1}}{T-\tau}\int_0^t\int_{B_R(x_0)}|w|^2dxds\nonumber\\
&\leq&\left(1+\frac{C_{1}T}{T-\tau}\right)\sup_{s\in[0,T]}\|u(\cdot,s)\|^2_{L^2(B_R(x_0))}+\frac{C_{1}}{T-\tau}\int_0^t\int_{B_R(x_0)}|w|^2dxds.
\end{eqnarray*}
 Here, we used (\ref{yu-6-6-6}). This, together with the uniform ellipticity condition (see (\ref{yu-11-28-2})), implies that
$$\| w(\cdot,t)\|_{L^2(B_R(x_0))}^2\leq \left(1+\frac{C_{1}T}{T-\tau}\right)\sup_{s\in[0,T]}\|u(\cdot,s)\|^2_{L^2(B_R(x_0))}+\frac{C_{1}}{T-\tau}\int_0^t\int_{B_R(x_0)}|w|^2dxds.$$
By the Gronwall inequality, we get (\ref{yu-6-7-6}) immediately.
 From (\ref{yu-6-7-6}) and the definition of $w$, we know that for any $t\in[0,T]$,
\begin{equation}\label{yu-6-7-11-1}
	\|v(\cdot,t)\|_{L^2(B_R(x_0))}\leq 2e^{\frac{C_{1}T}{T-\tau}}
	\sup_{s\in[0,T]}\|u(\cdot,s)\|_{L^2(B_R(x_0))}.
\end{equation}

\par
	We next claim that
\begin{eqnarray}\label{yu-6-8-3}
	\|\nabla w(\cdot,t)\|_{L^2(B_R(x_0))}^2
	\leq \frac{4C_{1}^{2}\lambda T}{(T-\tau)^{2}}
	\sup_{s\in[0,T]}\|u(\cdot,s)\|^2_{L^2(B_R(x_0))}+\lambda^{2}\|\nabla u_{0}\|_{L^2(B_R(x_0))}^2.
\end{eqnarray}
    If it holds true, along  with the definition of $w$, we have that, for each $t\in[0,T]$,
\begin{eqnarray}\label{yu-6-8-4}
	\|\nabla v(\cdot,t)\|_{L^2(B_R(x_0))}^2&\leq& 2\|\nabla w(\cdot,t)\|^2_{L^2(B_R(x_0))}+2\|\nabla u(\cdot,t)\|_{L^2(B_R(x_0))}^2\nonumber\\
&\leq&2\lambda^{2}\|\nabla u_{0}\|_{L^2(B_R(x_0))}^2+2\|\nabla u(\cdot,t)\|_{L^2(B_R(x_0))}^2\nonumber\\
&&+ \frac{8C_{1}^{2}\lambda T}{(T-\tau)^{2}}
	\sup_{s\in[0,T]}\|u(\cdot,s)\|^2_{L^2(B_R(x_0))}\nonumber\\
	&\leq& \left[2+2\lambda^{2}+\frac{8C_{1}^{2}\lambda T}{(T-\tau)^{2}}\right]\sup_{s\in[0,T]}\|u(\cdot,s)\|^2_{H^1(B_R(x_0))}.
\end{eqnarray}
Hence,  the desired estimate \eqref{yu-6-18-1} follows from $\lambda\geq1$, (\ref{yu-6-7-11-1}) and (\ref{yu-6-8-4})
when $t\in[0,T]$ with $\hat{C}_1=2(3+\lambda^2)(1+2C_1)$.

The rest is devoted to prove the claim  (\ref{yu-6-8-3}). In fact,
multiplying first (\ref{yu-6-6-7}) by $w_s$ and then integrating by parts  over $B_R(x_0)\times(0,t)$,
	we find
\begin{eqnarray*}\label{yu-6-7-12}
	&\;&\int_0^t\int_{B_R(x_0)}|w_s|^2dxds
	+\frac{1}{2}\int_0^t\int_{B_R(x_0)}[\nabla w\cdot(A\nabla w)]_sdxds\nonumber\\
&=&\int_0^t\int_{B_R(x_0)}|w_s|^2dxds+\frac{1}{2}\int_0^t\int_{B_R(x_0)}\left[(\nabla w)_{s}\cdot (A\nabla w)+\nabla w\cdot(A\nabla w)_{s}\right]dxds\nonumber\\
&=&\int_0^t\int_{B_R(x_0)}|w_s|^2dxds+\int_0^t\int_{B_R(x_0)}\nabla w_{s}\cdot(A\nabla w) dxds\nonumber\\
&=&\int_0^t\int_{B_R(x_0)}|w_s|^2dxds-\int_0^t\int_{B_R(x_0)} w_{s}\mathrm{div}(A\nabla w) dxds\nonumber\\
    &=&-\int_0^t\int_{B_R(x_0)}w_s\eta_sudxds\nonumber\\
&\leq&\frac{C_{1}}{\epsilon_{1}(T-\tau)}\int_0^t\int_{B_R(x_0)}
	|w_{s}|^2dxds+\frac{C_{1}\epsilon_{1}}{T-\tau}\int_0^t\int_{B_R(x_0)}|u|^2dxds\qquad \text{for any}\;\;\epsilon_{1}>0.
\end{eqnarray*}
	Here, we used (\ref{yu-6-6-6}). Letting
	$\epsilon_{1}=2C_{1}/(T-\tau)$
	in the inequality above, combined with the uniform ellipticity condition \eqref{yu-11-28-2},  we get that
\begin{eqnarray*}\label{yu-6-8-1}
	&\;&\int_{B_R(x_0)}\nabla w(x,t)\cdot(A(x)\nabla w(x,t))dx\nonumber\\
&\leq&\int_{B_R(x_0)}\nabla u_{0}\cdot(A(x)\nabla u_{0})dx+\frac{4C_{1}^{2}}{(T-\tau)^{2}}\int_0^t\int_{B_R(x_0)}|u|^2dxds\nonumber\\
&\leq&	\lambda\int_{B_R(x_0)}|\nabla u_{0}|^{2}dx+\frac{4C_{1}^{2}}{(T-\tau)^{2}}\int_0^t\int_{B_R(x_0)}|u|^2dxds.
\end{eqnarray*}	
By the uniform ellipticity condition \eqref{yu-11-28-2} again, this means that
\begin{eqnarray*}\label{yu-6-8-2}
	\|\nabla w(\cdot,t)\|_{L^2(B_R(x_0))}^2&\leq&
	\lambda^{2}\int_{B_R(x_0)}|\nabla u_{0}|^{2}dx+\frac{4\lambda C_{1}^{2}}{(T-\tau)^{2}}\int_0^t\int_{B_R(x_0)}|u|^2dxds\\
&\leq& \lambda^{2}\int_{B_R(x_0)}|\nabla u_{0}|^{2}dx+\frac{4C_{1}^{2}\lambda T}{(T-\tau)^{2}}
	\sup_{s\in[0,T]}\|u(\cdot,s)\|^2_{L^2(B_R(x_0))}.
\end{eqnarray*}
Thus,  (\ref{yu-6-8-3}) is true.

\medskip

\vskip 5pt
 \textit{Step 2. To prove \eqref{yu-6-18-1} when $t\geq T$.}
\par

 Denoting $\mathcal{A}(\cdot):=-\mbox{div}(A(\cdot)\nabla)$ with domain $D(\mathcal{A}):=H_{0}^1(B_R(x_0))\cap H^2(B_R(x_0))$,
we claim that
	there is a generic constant $C_{2}=C_{2}(\lambda,N)>0$ such that
 \begin{equation}\label{yu-6-7-9}
 	\langle\mathcal{A}f,f\rangle_{{L}^2(B_R(x_0))}\geq C_{2}R^{-2}\|f\|^2_{L^2(B_R(x_0))}
	\;\;\mbox{for each}\;\;f\in D(\mathcal{A}).
 \end{equation}
 	In fact, by the Poincar\'e inequality
\begin{equation*}\label{yu-11-30-b-1}
    \int_{B_R(x_0)}|f|^2dx\leq \left(\frac{2R}{N}\right)^2\int_{B_R(x_0)}|\nabla f|^2dx\;\;\mbox{for each}\;\;f\in H_0^1(B_R(x_0)),
\end{equation*}	
we derive
\begin{equation*}\label{yu-10-12-1}
	\langle\mathcal{A}f,f\rangle_{{L}^2(B_R(x_0))}\geq\lambda^{-1}\int_{B_R(x_0)}
	|\nabla f|^2dx\geq \lambda^{-1}\left(\frac{N}{2R}\right)^{2}
	\int_{B_R(x_0)}
	|f|^2dx.
\end{equation*}
 	We can conclude the claim (\ref{yu-6-7-9}).
\par

    As a consequence of \eqref{yu-6-7-9}, we see that the inverse of $(\mathcal A, D(\mathcal{A}))$ is positive, self-adjoint and compact in $L^2(B_R(x_0))$.
By the spectral theorem for compact self-adjoint operators,	there are  eigenvalues
	$\{\mu_i\}_{i\in\mathbb{N}^+}\subset \mathbb{R}^+$ and eigenfunctions
	$\{f_i\}_{i\in\mathbb{N}^+}\subset H_0^1(B_R(x_0))$, which make up a complete orthogonal basis of $L^2(B_R(x_0))$,
such that
 \begin{equation}\label{yu-6-7-10}
 \begin{cases}
 	-\mathcal{A}f_i=\mu_if_i\;\;\mbox{and}\;\;\|f_i\|_{{L}^2(B_R(x_0))}=1&\mbox{for each}\;\;i\in\mathbb{N}^+,\\
C_{2}R^{-2}\leq \mu_1\leq \mu_2\leq \cdots
         \leq \mu_i\to+\infty&\mbox{as}\;\;i\to+\infty.
 \end{cases}
 \end{equation}
 Then, by the formula of Fourier decomposition,
	the solution $w$  of (\ref{yu-6-6-7}) in $[T,+\infty)$ is given by
\begin{equation*}\label{yu-6-12-3}
	w(\cdot,t)=\sum_{i=1}^\infty\langle w(\cdot,T),f_i\rangle_{{L}^2(B_R(x_0))}e^{-\mu_i(t-T)}f_i	\quad\text{in}\;\;B_R(x_0)\;\;\mbox{for each}\;\;t\in[T,+\infty).
\end{equation*}
    	Hence, we deduce that for each $t\in[T,+\infty)$,
\begin{eqnarray}\label{yu-6-12-4}
	\|w(\cdot,t)\|^2_{{L}^2(B_R(x_0))}
	\leq e^{-C_{2}R^{-2}(t-T)}\|w(\cdot,T)\|_{{L}^2(B_R(x_0))}^2
\end{eqnarray}
   and
\begin{eqnarray*}
	w_t(\cdot,t)=-\sum_{i=1}^\infty\mu_i\langle w(\cdot,T),f_i\rangle_{{L}^2(B_R(x_0))}e^{-\mu_i(t-T)}f_i.
\end{eqnarray*}
        It follows that
\begin{eqnarray}\label{yu-6-13-1}
	-\langle w(\cdot,t),w_t(\cdot,t)\rangle_{{L}^2(B_R(x_0))}
	=\sum_{i=1}^\infty\mu_i|\langle w(\cdot,T),f_i\rangle_{{L}^2(B_R(x_0))}|^2
	e^{-2\mu_i(t-T)},
\end{eqnarray}
for each $t\in[T,+\infty)$. In particular, by taking $t=T$ in the above identify, we have
\begin{equation}\label{yu-6-13-2}
	-\langle w(\cdot,T),w_t(\cdot,T)\rangle_{{L}^2(B_R(x_0))}
	=\sum_{i=1}^\infty\mu_i|\langle w(\cdot,T),f_i\rangle_{{L}^2(B_R(x_0))}|^2.	
\end{equation}
    Meanwhile, it follows from (\ref{yu-6-6-7}), (\ref{yu-6-6-6}) and \eqref{yu-11-28-2} that
\begin{eqnarray}\label{yu-6-13-3}
-\langle w(\cdot,T),w_t(\cdot,T)\rangle_{{L}^2(B_R(x_0))}\leq \lambda\|\nabla w(\cdot,T)\|^2_{L^2(B_R(x_0))}.
\end{eqnarray}
    	From (\ref{yu-6-13-2}) and (\ref{yu-6-13-3}), we have
\begin{equation*}\label{yu-6-14-1}
	\sum_{i=1}^\infty\mu_i|\langle w(\cdot,T),f_i\rangle_{{L}^2(B_R(x_0))}|^2
	\leq \lambda\|\nabla w(\cdot,T)\|^2_{L^2(B_R(x_0))}.
\end{equation*}
    	This, together with  (\ref{yu-6-13-1}) and (\ref{yu-6-7-10}), gives
\begin{equation}\label{yu-6-14-2}
	-\langle w(\cdot,t),w_t(\cdot,t)\rangle_{{L}^2(B_R(x_0))}
	\leq \lambda e^{-2C_{2}R^{-2}(t-T)}\|\nabla w(\cdot,T)\|^2_{L^2(B_R(x_0))},
\end{equation}
	for each $t\in[T,+\infty)$. On the other hand, by (\ref{yu-6-6-6}), (\ref{yu-6-6-7}) and \eqref{yu-11-28-2}, we see that for each $t\in[T,+\infty)$,
 \begin{equation}\label{yu-6-14-3}
 	-\langle w(\cdot,t),w_t(\cdot,t)\rangle_{{L}^2(B_R(x_0))}=\langle w(\cdot,t),\mathcal{A}w(\cdot,t)\rangle_{{L}^2(B_R(x_0))}
	\geq \lambda^{-1}\|\nabla w(\cdot,t)\|^2_{L^2(B_R(x_0))}.
 \end{equation}
    	By (\ref{yu-6-14-2}) and (\ref{yu-6-14-3}), we find that for each $t\in[T,\infty)$,
 \begin{equation*}\label{yu-6-14-4}
 	\|\nabla w(\cdot,t)\|^2_{L^2(B_R(x_0))}
	\leq \lambda^{2}
	e^{-2C_{2}R^{-2}(t-T)}\|\nabla w(\cdot,T)\|^2_{L^2(B_R(x_0))}.
\end{equation*}
    	This, together with (\ref{yu-6-12-4}) and $\lambda\geq1$, means that
\begin{equation*}\label{yu-6-18-2}
	\|w(\cdot,t)\|_{H^1(B_R(x_0))}^2\leq \lambda^{2}e^{-C_{2}R^{-2}(t-T)}\|w(\cdot,T)\|^2_{H^1(B_R(x_0))}.
\end{equation*}
By (\ref{yu-6-6-6}), we have  that  $w(\cdot,t)=v(\cdot,t)$ for each $t\geq T$. Thus, by (\ref{yu-6-18-1}) for the case $t\in[0,T]$, we conclude the desired result with
$\hat{C}_1=2\lambda^2(3+\lambda^2)(1+2C_1)$ and $\hat{C}_2=C_{2}/2$.
\end{proof}

Define
\begin{equation*}\label{yu-6-18-5}
	\tilde{v}(\cdot,t):=
\begin{cases}
	v(\cdot,t)&\mbox{if}\;\;t\geq 0,\\
	0&\mbox{if}\;\;t<0,
\end{cases}
\end{equation*}
where $v$ is the solution of \eqref{yu-11-29-4}.
By Lemma \ref{yu-lemma-6-10-1}, we see the Fourier transform of	$\tilde{v}$ with respect to the time variable $t\in\mathbb R$ is meaningful
\begin{equation*}\label{yu-6-18-6}
	\hat{v}(x,\mu)=\int_{\mathbb{R}}e^{-i\mu t}\tilde{v}(x,t)dt\quad\text{for}\;\;(x,\mu)\in B_R(x_0)\times\mathbb R.
\end{equation*}
We further have
\begin{lemma}\label{yu-lemma-6-18-1}
There exist positive constants $\hat{C}_{3}=\hat{C}_{3}(\lambda,N)$ and $\hat{C}_{4}=\hat{C}_{4}(\lambda,N)$ such that  for each  $\mu\in\mathbb{R}$,
the following two estimates hold:
\begin{equation}\label{yu-6-23-5}
	\|\nabla \hat{v}(\cdot,\mu)\|^{2}_{L^2(B_r(x_0))}\leq \frac{\hat{C}_{3}}{(R-2r)^{2}}\|\hat{v}(\cdot,\mu)\|^{2}_{L^2(B_{\frac{R}{2}}(x_0))}\quad \text{for all} \;\;0<r <\frac{R}{2}
\end{equation}
and
\begin{equation}\label{yu-6-22-16}
	\|\hat{v}(\cdot,\mu)\|_{L^2(B_{\frac{R}{2}}(x_0))}\leq \hat{C}_{4}
	\left(T+R^{2}\right)e^{\frac{\hat{C}_{1}T}{T-\tau}-\frac{\sqrt{|\mu|}R}{4e\hat{C}_{5}}}
	F(R),
\end{equation}
where $\hat{C}_{5}=\sqrt{8\pi^2\sqrt{2+\lambda^{6}}}$, $\hat{C}_1>0$ and $F(R)$ are given by Lemma \ref{yu-lemma-6-10-1}.
\end{lemma}
\begin{proof}
By \eqref{yu-11-29-4}, we have that for each $\mu\in\mathbb R$,
\begin{equation}\label{yu-6-18-7}
	i\mu\hat{v}(x,\mu)-\mbox{div}(A(x)\nabla \hat{v}(x,\mu))=0
	\;\;\mbox{in}\;\;B_{R}(x_0).
\end{equation}
Take arbitrarily $r\in(0,R/2)$ and define a cutoff function $\psi \in C^\infty(\mathbb{R}^N;[0,1])$ verifying
\begin{equation}\label{yu-6-23-1}
\begin{cases}
	\psi=1&\mbox{in}\;\;\overline{B_r(x_0)},\\
	\psi=0&\mbox{in}\;\;\mathbb{R}^N\backslash B_{\frac{R}{2}}(x_0),\\
	|\nabla\psi|\leq \frac{C_{3}}{R-2r}&\mbox{in}\;\;\mathbb{R}^N,
\end{cases}
\end{equation}
where $C_{3}=C_{3}(N)$ is a positive  constant. The rest proof is divided into two steps.
\vskip 5pt
    \textit{Step 1. The proof of (\ref{yu-6-23-5}).}
\par

Multiplying first (\ref{yu-6-18-7}) by $\bar{\hat{v}}\psi^2$ and then integrating by parts  over
	$B_{\frac{R}{2}}(x_0)$, we have
\begin{eqnarray*}\label{yu-6-23-2}
        \int_{B_{\frac{R}{2}}(x_0)}\nabla\bar{\hat{v}}\cdot (A\nabla\hat{v})\psi^2dx
	+2\int_{B_{\frac{R}{2}}(x_0)}\nabla\psi\cdot(A\nabla\hat{v})\bar{\hat{v}}\psi dx
	=-i\int_{B_{\frac{R}{2}}(x_0)}\mu |\hat{v}|^2\psi^2dx.
\end{eqnarray*}
By (\ref{yu-11-28-2}) and the Young inequality, we derive that for each $\epsilon_2>0$,
\begin{eqnarray}\label{yu-6-23-3}
	&\;&\lambda^{-1}\int_{B_{\frac{R}{2}}(x_0)}|\nabla\hat{v}|^2\psi^2dx
=\lambda^{-1}\mbox{Re}\int_{B_{\frac{R}{2}}(x_0)}|\nabla\hat{v}|^2\psi^2dx\nonumber\\
	&\leq&-2\mbox{Re}
	\int_{B_{\frac{R}{2}}(x_0)}\nabla\psi\cdot(A\nabla\hat{v})\bar{\hat{v}}\psi dx
\leq 2\lambda
	\int_{B_{\frac{R}{2}}(x_0)}|\nabla\psi||\nabla\hat{v}||\bar{\hat{v}}||\psi| dx\nonumber\\
	&\leq&\epsilon_2\int_{B_{\frac{R}{2}}(x_0)}|\nabla \hat{v}|^2\psi^2dx
	+\frac{\lambda^2}{\epsilon_2}\int_{B_{\frac{R}{2}}(x_0)}|\hat{v}|^2|\nabla\psi|^2dx.	
\end{eqnarray}
Taking $\epsilon_2=1/(2\lambda)$  in the above inequality, by (\ref{yu-6-23-1}) and (\ref{yu-6-23-3}), we derive
 (\ref{yu-6-23-5}).

\vskip 5pt
    \textit{Step 2. The proof of (\ref{yu-6-22-16}).}
\par

Note that when $\mu=0$, by (\ref{yu-6-7-11-1}), (\ref{yu-6-12-4}) and the definitions of $\hat{v}, \tilde{v}, w$,  we have
\begin{eqnarray}\label{yu-6-22-15}
&&\|\hat{v}(\cdot,0)\|_{L^2(B_{R}(x_0))}\nonumber\\
&=&
\left\|\int_{\mathbb{R}}\tilde{v}(x,t)dt\right\|_{L^2(B_{R}(x_0))}
=\left\|\int_{0}^{+\infty}v(x,t)dt\right\|_{L^2(B_{R}(x_0))}\nonumber\\
	&\leq&\int_{0}^{T}\|v(x,t)\|_{L^2(B_{R}(x_0))}dt+\int_{T}^{+\infty}\|v(x,t)\|_{L^2(B_{R}(x_0))}dt\nonumber\\
&\leq&2Te^{\frac{C_{1}T}{T-\tau}}
	\sup_{s\in[0,T]}\|u(\cdot,s)\|_{L^2(B_R(x_0))}+\int_{T}^{+\infty}e^{-\frac{C_{2}(t-T)}{2R^2}}\|v(T)\|_{L^2(B_{R}(x_0))}dt\nonumber\\
	&\leq&2e^{\frac{C_{1}T}{T-\tau}}
	\sup_{s\in[0,T]}\|u(\cdot,s)\|_{H^1(B_R(x_0))}\left[T+\int_{T}^{+\infty}e^{-\frac{C_{2}(t-T)}{2R^2}}dt\right]\nonumber\\
&\leq&2e^{\frac{C_{1}T}{T-\tau}}
	\left[T+\frac{2R^{2}}{C_{2}}\right]\sup_{s\in[0,T]}\|u(\cdot,s)\|_{H^1(B_R(x_0))}.
\end{eqnarray}
    This yields that \eqref{yu-6-22-16} is true when $\mu=0$. Thus it suffices to prove \eqref{yu-6-22-16} in the case that $\mu\neq0$.
To this end,  define for each $\mu\in\mathbb{R}\setminus\{0\}$,
\begin{equation*}\label{yu-6-19-1}
	p(x,\xi,\mu)=e^{i\sqrt{|\mu|}\xi}\hat{v}(x,\mu) \;\; \mbox{for a.e.}\;\; (x,\xi)\in B_R(x_0)\times\mathbb R.
\end{equation*}
Then,  by (\ref{yu-6-18-7}), one can  check easily that, for each fixed $\mu\in\mathbb{R}\setminus\{0\}$, $p(\cdot,\cdot,\mu)$ verifies
\begin{equation*}\label{yu-6-19-2}
	\mbox{div}(A(x)\nabla p(x,\xi,\mu))+i\mbox{sign}(\mu)\partial_{\xi\xi}p(x,\xi,\mu)
	=0\;\;\mbox{in}\;\;B_R(x_0)\times\mathbb{R}.
\end{equation*}
Here
\begin{equation*}\label{yu-6-19-3}
     \mbox{sign}(\mu)=
\begin{cases}
	1&\mbox{if}\;\;\mu>0,\\
	-1&\mbox{if}\;\;\mu<0.
\end{cases}
\end{equation*}

As for the equation verified by $p(\cdot,\cdot,\mu)$, it is elliptic with complex coefficients and its coefficients are independent of the $\xi$-variable. These facts imply by standard energy methods  ($m$ times localized Cacciopoli’s inequalities, see also \cite{Canuto-Rosset-Vessella}) that
\begin{equation}\label{yu-6-22-4}
	\int_{B_{\frac{R}{2}}(x_0)\times(-\frac{R}{2},\frac{R}{2})}|p^{(m)} |^2dxd\xi
	\leq2R\left[\frac{C_{4}m^2}{R^2}\right]^m\int_{B_R(x_0)}|\hat{v}(x,\mu)|^2dx
\end{equation}
with $p^{(m)}=\partial_\xi^m p$ and $C_{4}=8\pi^2\sqrt{2+\lambda^{6}}$ (For the sake of completion, a detail proof for \eqref{yu-6-22-4} is presented in Appendix).
By the similar  arguments in (\ref{yu-6-22-15}) and the definition of $\hat{v}$, we get that  for each $\mu\in\mathbb{R}$,
\begin{eqnarray}\label{yu-6-22-5}
	\|\hat{v}(\cdot,\mu)\|_{L^2(B_{R}(x_0))}
&\leq&\int_{0}^{+\infty}\|v(x,t)\|_{L^2(B_{R}(x_0))}dt\nonumber\\
	&\leq&2e^{\frac{C_{1}T}{T-\tau}}
	\left[T+\frac{2R^{2}}{C_{2}}\right]\sup_{s\in[0,T]}\|u(\cdot,s)\|_{H^1(B_R(x_0))}.
\end{eqnarray}
	Therefore, by (\ref{yu-6-22-4}), (\ref{yu-6-22-5}) and the definition of $F(R)$, we get that for each $m\in\mathbb{N}^+$,
\begin{eqnarray}\label{yu-6-22-6}
	\int_{B_{\frac{R}{2}}(x_0)\times(-\frac{R}{2},\frac{R}{2})}|p^{(m)}|^2dxd\xi
	\leq8R \left[\frac{C_{4}m^2}{R^2}\right]^m e^{\frac{2C_{1}T}{T-\tau}}
	\left[T+\frac{2R^{2}}{C_{2}}\right]^{2}
	F^2(R).
\end{eqnarray}
\par
	For any $\varphi\in L^2(B_{R/2}(x_0);\mathbb{C})$, we define
\begin{equation*}\label{yu-6-22-7}
	P_{\mu}(\xi):=\int_{B_{\frac{R}{2}}(x_0)}p(x,\xi,\mu)\bar{\varphi}(x)dx, \;\;\;\xi\in\left(-\frac{R}{2},\frac{R}{2}\right).
\end{equation*}
	It is well known that the following interpolation inequality holds
\begin{equation}\label{yu-6-22-8}
	\|f\|_{L^\infty(I)}\leq C_{5}\left(|I|\|f'\|_{L^2(I)}^2+\frac{1}{|I|}\|f\|_{L^2(I)}^2\right)^{\frac{1}{2}}
	\;\;\mbox{for each}\;\;f\in H^1(I),
\end{equation}
	where  $C_{5}>0$, $I$ is an bounded nonempty interval of $\mathbb{R}$ and $|I|$ is the length. Therefore, by
	(\ref{yu-6-22-6}) and (\ref{yu-6-22-8}),  we have that for any $\xi\in(-R/2, R/2)$ and $m\in\mathbb{N}^+$,
\begin{eqnarray}\label{yu-6-22-9}
	&&|P_{\mu}^{(m)}(\xi)|\nonumber\\
&\leq& C_{5}\left(R\int_{-\frac{R}{2}}^{\frac{R}{2}}|P_{\mu}^{(m+1)}(\xi)|^2d\xi
	+\frac{1}{R}\int_{-\frac{R}{2}}^{\frac{R}{2}}|P_{\mu}^{(m)}(\xi)|^2d\xi\right)^{\frac{1}{2}}\nonumber\\
&=&C_{5}\left(R\int_{-\frac{R}{2}}^{\frac{R}{2}}\left|\int_{B_{\frac{R}{2}}(x_0)}p^{(m+1)}(x,\xi,\mu)\bar{\varphi}(x)dx\right|^2d\xi
	+\frac{1}{R}\int_{-\frac{R}{2}}^{\frac{R}{2}}\left|\int_{B_{\frac{R}{2}}(x_0)}p^{(m)}(x,\xi,\mu)\bar{\varphi}(x)dx\right|^2d\xi\right)^{\frac{1}{2}}\nonumber\\
	&\leq&C_{5}\left(R\int_{B_{\frac{R}{2}}(x_0)\times(-\frac{R}{2},\frac{R}{2})}
	|p^{(m+1)}|^2dxd\xi+\frac{1}{R}\int_{B_{\frac{R}{2}}(x_0)\times(-\frac{R}{2},\frac{R}{2})}|p^{(m)}|^2dxd\xi\right)
	^{\frac{1}{2}}\|\varphi\|_{L^2(B_{\frac{R}{2}}(x_0))}\nonumber\\
	&\leq&4C_{5}e^{\frac{C_{1}T}{T-\tau}}
	\left[T+\frac{2R^{2}}{C_{2}}\right] F(R)\frac{C_{4}^{\frac{m+1}{2}}(m+1)^{m+1}}{R^{m}}\|\varphi\|_{L^2(B_{\frac{R}{2}}(x_0))}.
\end{eqnarray}
	This implies that  $P_{\mu}(\cdot)$ can be
	analytically extended to the complex ball (still denoted by the same notation)
\begin{equation*}\label{yu-6-22-10}
	E_0:=\left\{\xi\in\mathbb{C}:|\xi|<\frac{R}{2\sqrt{C_{4}}e} \right\}.
\end{equation*}
Then,
\begin{equation*}\label{yu-6-22-11}
	|P_{\mu}(\xi)|\leq \sum_{m=0}^\infty\frac{|P_{\mu}^{(m)}(0)|}{m!}|\xi|^m,
\end{equation*}
when $\xi\in i\mathbb{R}\cap E_0$.
Taking $\xi_0=-iR/(4\sqrt{C_{4}}e)$, by (\ref{yu-6-22-9}), we  get that
\begin{equation}\label{yu-6-22-12}
	|P_{\mu}(\xi_0)|\leq 4C_{5}\sqrt{C_{4}}e^{\frac{C_{1}T}{T-\tau}}
	\left[T+\frac{2R^{2}}{C_{2}}\right]
	F(R)\sum_{m=0}^\infty\frac{(m+1)^{m+1}}{m!(4e)^m}
	\|\varphi\|_{L^2(B_{\frac{R}{2}}(x_0))}.
\end{equation}
While, by the definitions of $P_{\mu}(\cdot)$ and $\xi_0$,
\begin{eqnarray*}\label{yu-6-22-13}
	P_{\mu}(\xi_0)=e^{\frac{\sqrt{|\mu|}R}{4\sqrt{C_{4}}e}}\int_{B_{\frac{R}{2}}(x_0)}\hat{v}(x,\mu)\bar{\varphi}(x)dx.
\end{eqnarray*}
	This, together with (\ref{yu-6-22-12}) and the arbitrariness of $\varphi$, means that,
	\begin{equation}\label{yu-6-22-14}
	\|\hat{v}(\cdot,\mu)\|_{L^2(B_{\frac{R}{2}}(x_0))}\leq  4C_{5}\sqrt{C_{4}}
	\left[T+\frac{2R^{2}}{C_{2}}\right]e^{\frac{C_{1}T}{T-\tau}-\frac{\sqrt{|\mu|}R}{4\sqrt{C_{4}}e}}
\sum_{m=0}^\infty\frac{(m+1)^{m+1}}{m!(4e)^m}
	F(R).
\end{equation}
		By (\ref{yu-6-22-14}) and (\ref{yu-6-22-15}), we derive (\ref{yu-6-22-16}) and complete the proof.	
\end{proof}

\medskip

\subsection{Stability estimate for elliptic equations}\label{yu-section-7-26-3}

For our purpose and the convenience of the reader, we first quote from \cite[Theorem 5.1]{Logunov-Malinnikova} the following propagation estimate of smallness (for sets on hyperplanes) for gradients of solutions to elliptic equations in the way we understood.

\begin{lemma}\label{auxiliary-lemma}
Let $0<r<r'<R<+\infty$ and $\rho>0.$ Suppose a positive measurable set $\omega\subset B_{r/2}(x_0)$  satisfy $|\omega|/|B_{r}(x_0)|>\rho$. If $g$ is a solution to $\mathrm{div}(A(x)\nabla g)+g_{x_{N+1}x_{N+1}}=0$ in $B_R(x_0,0),$ with $A(\cdot)$ satisfying the condition (\ref{yu-11-28-2}), then there exists $\hat{C}_{6}=\hat{C}_{6}(r,r'\rho,\lambda)>0$ and  $\alpha_{1}=\alpha_{1}(r,r'\rho,\lambda)\in(0,1)$  such that

{\begin{equation}\label{auxiliary-lemma11}
	\sup_{B_{r}(x_0,0)}|(\nabla g,\partial_{N+1}g)|\leq \hat{C}_{6}\left(\sup_{\omega}| (\nabla g(\cdot,0),\partial_{N+1}g(\cdot,0))|\right)^{\alpha_{1}}\left(\sup_{B_{r'}(x_0,0)}|(\nabla g,\partial_{N+1}g)|\right)^{1-\alpha_{1}}.
\end{equation}}
\end{lemma}

\begin{remark}
We refer the interesting reader to \cite{Logunov-Malinnikova}  for more general statements on propagation of
smallness on possibly lower dimensional subsets for solutions to elliptic equations.
\end{remark}

\begin{remark}
Regarding the dependence of $\omega$, we here emphasize that the constants $\hat{C}_{6}$ and $\alpha_{1}$ in (\ref{auxiliary-lemma11})
depend only on the $r, r'$ and $N$-dimensional Lebesgue measure $|\omega|$, but not on the shape or position of $\omega$. This
uniform dependence is very important in our approach to prove the main result of this paper (see
also \cite{Apraiz-Escauriaza-Wang-Zhang,Escauriaza-Montaner-Zhang,Escauriaza-Montaner-Zhang2}).
\end{remark}

Let $0<2r<R<+\infty$ and $\rho>0.$ Suppose a positive measurable set $\omega\subset B_{r/2}(x_0)$  satisfy $|\omega|/|B_{r}(x_0)|>\rho$.
Let $g\in H^1(B_R(x_0,0))$ be a solution of
\begin{equation}\label{yu-6-23-9}
\begin{cases}
	\mbox{div}(A(x)\nabla g)+g_{x_{N+1}x_{N+1}}=0&\;\;\;\text{in}\;\; B_R(x_0,0),\\
	g(x,0)=0&\;\;\;\text{in}\;\; B_R(x_0),\\
	g_{x_{N+1}}(x,0)=z(x)&\;\;\;\text{in}\;\; B_R(x_0),
\end{cases}
\end{equation}
	where $z\in L^2(B_R(x_0))$.

\begin{lemma}\label{yu-proposition-7-1-1}
There exist constants $\hat{C}_{7}=\hat{C}_{7}(N,\rho,r,\lambda)>0$ and  $\alpha_{2}=\alpha_{2}(N,\rho,r,\lambda)\in(0,1)$  such that
{\begin{equation}\label{yu-7-3-b-2}
	\|g\|_{H^1(B_r(x_0,0))}\leq \hat{C}_{7}\|g\|_{H^1(B_{\frac{3r}{2}}(x_0,0))}^{\alpha_{2}}\|z\|^{1-\alpha_{2}}_{L^2(\omega)}.
\end{equation}}
\end{lemma}	
\begin{proof}
We quote from \cite[Lemma 4.3]{Lin-1991}, \cite[Lemma A]{Lebeau-Zuazua} and \cite[Theorem 5.1]{LU-2013} that the following stability estimate holds: there exist constants $C_{6}=C_{6}(\lambda,N)>0$ and $\gamma=\gamma(r,\lambda,N)\in(0,1)$ such that
{\begin{equation}\label{yu-7-3-b-2000}
	\|g\|_{H^1(B_r(x_0,0))}\leq C_{6}r^{-2}\| g\|_{H^1(B_{\frac{7r}{6}}(x_0,0))}^{\gamma}\|z\|^{1-\gamma}_{L^2(B_{\frac{7r}{6}}(x_0))}.
\end{equation}}
We first assume $z\in C(\overline{B_R(x_0)})$.
 Applying Lemma \ref{auxiliary-lemma} to gradients of solutions for (\ref{yu-6-23-9}), we have
{\begin{equation}\label{yu-7-3-b-2000111}
	\|(\nabla g,\partial_{N+1}g)\|_{L^\infty(B_{\frac{7r}{6}}(x_0,0))}\leq C_{7}\| (\nabla g(\cdot,0),\partial_{N+1}g(\cdot,0))\|_{L^\infty(\omega)}^{\theta}\|(\nabla g,\partial_{N+1}g)\|^{1-\theta}_{L^{\infty}(B_{\frac{4r}{3}}(x_0,0))}
\end{equation}}
with $C_{7}=C_{7}(\rho,r,\lambda)>0$ and  $\theta=\theta(\rho,r,\lambda)\in(0,1).$

Next, we observe from the second line in (\ref{yu-6-23-9}) that $\nabla g(\cdot,0)=0$ in $ B_{R}(x_0)$. By (\ref{yu-7-3-b-2000111}) and the standard elliptic regularity (see, e.g., \cite[Theorem 8.32]{Gilbarg-Trudinger})
$$\|(\nabla g,\partial_{N+1}g)\|_{L^\infty(B_{\frac{4r}{3}}(x_0,0))}\leq C_{8}\| g\|_{H^1(B_{\frac{3r}{2}}(x_0,0))}$$
with $C_{8}=C_{8}(r,\lambda,N)>0$,
we have
$$\|z\|_{L^\infty(B_{\frac{7r}{6}}(x_0))}\leq \|(\nabla g,\partial_{N+1}g)\|_{L^\infty(B_{\frac{7r}{6}}(x_0,0))}\leq C_{7}C^{1-\theta}_{8}\|z\|_{L^\infty(\omega)}^{\theta}\| g\|^{1-\theta}_{H^1(B_{\frac{3r}{2}}(x_0,0))}.$$
This, combined with (\ref{yu-7-3-b-2000}), leads to
\begin{eqnarray}\label{yu-7-3-b-2000222}
	\|g\|_{H^1(B_r(x_0,0))}&\leq& C_{6}r^{-2}\| g\|_{H^1(B_{\frac{7r}{6}}(x_0,0))}^{\gamma}\|z\|^{1-\gamma}_{L^2(B_{\frac{7r}{6}}(x_0))}\nonumber\\
&\leq& C_{6}r^{-2}|B_{\frac{7r}{6}}|^{1-\gamma}\|g\|_{H^1(B_{\frac{7r}{6}}(x_0,0))}^{\gamma}\|z\|^{1-\gamma}_{L^\infty(B_{\frac{7r}{6}}(x_0))}\nonumber\\
&\leq& C_{6}r^{-2}|B_{\frac{7r}{6}}|^{1-\gamma}\|g\|_{H^1(B_{\frac{4}{3}r}(x_0,0))}^{\gamma}\left[C_{7}C^{1-\theta}_{8}\|z\|_{L^\infty(\omega)}^{\theta}\| g\|^{1-\theta}_{H^1(B_{\frac{3r}{2}}(x_0,0))}\right]^{1-\gamma}\nonumber\\
	&\leq& C_{9}\|z\|_{L^\infty(\omega)}^{\beta}\| g\|^{1-\beta}_{H^1(B_{\frac{3r}{2}}(x_0,0))}
\end{eqnarray}
with $C_{9}=C_{9}(r,\rho,N,\lambda)>0$ and  $\beta=\beta(r,\rho,N,\lambda)\in(0,1).$

Finally, we complete the proof by replacing the $L^\infty$-norm in (\ref{yu-7-3-b-2000222}) with $L^2$-norm. To this end, we define a new Lebesgue measurable set
{\begin{equation}\label{yu-7-3-b-2000333}
\tilde{\omega}:=\left\{x\in \omega: \left|z(x)\right|\leq\left(\frac{2}{|\omega|}\int_{\omega}|z|^{2}dx\right)^{\frac{1}{2}}\right\}.
\end{equation}}
It is clear that $$\int_{\omega}|z|^{2}dx\geq\int_{\omega\backslash \tilde{\omega}}|z|^{2}dx\geq\frac{2|\omega\backslash \tilde{\omega}|}{|\omega|}\int_{\omega}|z|^{2}dx,$$
where $\omega\backslash \tilde{\omega}=\omega\cap \tilde{\omega}^{c}$.
This implies that $|\omega\backslash \tilde{\omega}|\leq |\omega|/2$, and hence $\tilde{\omega}\geq |\omega|/ 2$ . Applying (\ref{yu-7-3-b-2000222}) with $\omega$ replaced by $\tilde{\omega}$
leads to
$$\|g\|_{H^1(B_r(x_0,0))}\leq C_{10}\| g\|_{H^1(B_{\frac{3r}{2}}(x_0,0))}^{\beta_{1}}\|z\|^{1-\beta_{1}}_{L^\infty(\tilde{\omega})}$$
for some new constant $C_{10}=C_{10}(r,\rho,N,\lambda)>0$ and  $\beta_{1}=\beta_{1}(r,\rho,N,\lambda)\in(0,1).$
This, together with (\ref{yu-7-3-b-2000333}), indicates the desired estimate
(\ref{yu-7-3-b-2}) for the case $z\in C(\overline{B_R(x_0)})$. This, together with (\ref{yu-7-3-b-2000}) and the fact
 $C(\overline{B_R(x_0)})$ is dense in $L^2(B_R(x_0))$, yields  that
 (\ref{yu-7-3-b-2}) holds for any $z\in L^2(B_R(x_0))$. This ends the proof.
\end{proof}

\section{Proof of Theorem \ref{yu-theorem-7-10-6}}\label{kaodu3}

In order to give the proof of Theorem \ref{yu-theorem-7-10-6}, we need  the following the interpolation inequality (i.e.,  Lemma \ref{lemma-2A2}), whose proof below is based on a reduction method \cite{LO} (see also \cite{Lin90}).  It is worth mentioning that  the interpolation inequality established in Lemma  \ref{lemma-2A2}
is  related to  the so-called two spheres and one cylinder inequality obtained  in \cite[Theorem 3.1.1$'$]{Canuto-Rosset-Vessella} (see also \cite[Theorem 2]{Escauriaza-Fernandez-Vessella-2006}) which has some applications in the study of inverse problem
 (see, for instance, \cite{Canuto-Rosset-Vessella}).  Comparing  with them, however, there are two main differences: $(i)$ the radius $r$ of the observation sphere is here allowed to be independent of the observation time $\tau$; $(ii)$ we  specify the dependence of the observability constant on the observation time $\tau$.

\begin{lemma}\label{lemma-2A2}
Let $T>0$, $0<r<+\infty$ and $\rho>0$. Suppose a positive measurable set $\omega\subset B_{r/2}(x_0)$  satisfy $|\omega|/|B_{r}(x_0)|>\rho$.  There exist constants $\hat{C}_{8}=\hat{C}_{8}(r,N,\lambda,\rho)>0$ and  $\sigma=\sigma(r,N,\lambda,\rho)\in(0,1)$ such that for each
$\tau\in(0,T/2)$, the corresponding solution  $u$ of \eqref{yu-6-24-1} with the initial value $u_{0}\in H^{1}(\mathbb R^N)$ satisfies
\begin{equation*}\label{yu-7-10-2}
	\|u(\cdot,\tau)\|_{L^2(B_r(x_0))}\leq \hat{C}_{8}
	\left[T^{2}e^{\frac{\hat{C}_{1}T}{T-\tau}}+e^{\frac{\hat{C}_{8}}{\tau}}\right]\|u(\cdot,\tau)\|_{L^2(\omega)}^\sigma
	\left(\sup_{s\in[0,T]}\|u(\cdot,s)\|^{2}_{H^1(B_{\varrho r}(x_0))}\right)^{\frac{1-\sigma}{2}},
	\end{equation*}
where $\varrho:=16\hat{C}_{5}e$, the constant $\hat{C}_1>0$ is given by Lemma \ref{yu-lemma-6-10-1} and the constant $\hat{C}_5=\sqrt{8\pi^2\sqrt{2+\lambda^6}}$ is given by Lemma \ref{yu-lemma-6-18-1}.
\end{lemma}
\begin{proof}
Taking $u_0\in H^1(\mathbb{R}^N)$ arbitrarily. Let $u$ be the solution to the equation (\ref{yu-6-24-1})  with the initial value $u_{0}\in H^{1}(\mathbb R^N)$, and $R:=\varrho r$.
Let $u_1$ and $u_2$ be accordingly the solutions to
\begin{equation*}\label{yu-7-4-4}
\begin{cases}
	\partial_tu_{1}-\mbox{div}(A(x)\nabla u_1)=0&\mbox{in} \;B_R(x_0)\times(0,2T),\\
	u_1=u&\mbox{on}\;\;\partial B_R(x_0)\times(0,2T),\\
	u_1(\cdot,0)=0 &\mbox{in}\;\;B_{R}(x_0)
\end{cases}
\end{equation*}	
	and
\begin{equation*}\label{yu-7-4-5}
\begin{cases}
	\partial_tu_{2}-\mbox{div}(A(x)\nabla u_2)=0&\mbox{in}\;\;  B_R(x_0)\times(0,2T),\\
	u_2=0&\mbox{on}\;\;\partial B_R(x_0)\times(0,2T),\\
	u_2(\cdot,0)=u_0&\mbox{in}\;\; B_{R}(x_0).
\end{cases}
\end{equation*}
It is clear that $u=u_1+u_2$ in $B_R(x_0)\times[0,2T]$.
By the standard energy estimate for solutions of parabolic equations, we have
\begin{equation}\label{yu-7-4-7}
	\sup_{t\in[0,T]}\|u_2(\cdot,t)\|_{H^1(B_R(x_0))}\leq C_{11} \|u_0\|_{H^1(B_R(x_0))}
\end{equation}
with $C_{11}=C_{11}(N,\lambda)>0.$

Arbitrarily fix  $\tau\in(0,T/2)$. Let $v_1$ be the solution to
 \begin{equation*}\label{yu-11-29-4-jia}
\begin{cases}
    \partial_tv_1-\mbox{div}(A(x)\nabla v_1)=0&\mbox{in}\;\;B_{R}(x_0)\times\mathbb{R}^+,\\
    v_1=\eta u&\mbox{on}\;\;\partial B_R(x_0)\times\mathbb{R}^+,\\
    v_1(\cdot,0)=0&\mbox{in}\;\; B_R(x_0),
\end{cases}
\end{equation*}
where $\eta$ is given by \eqref{yu-6-6-6}.
It is clear that
	 $u=v_1+u_2$ in $B_R(x_0)\times[0,\tau]$.
	 We extend $v_1$ to $\mathbb R^-\times B_R(x_0)$ by zero, and still denote it by the same way.
Define
\begin{equation*}\label{yu-6-18-6jia}
	\hat{v}_1(x,\mu)=\int_{\mathbb{R}}e^{-i\mu t}v_1(x,t)dt
	\quad\text{for}\;\;(x,\mu)\in B_R(x_0)\times\mathbb R.
\end{equation*}
Note from Lemma \ref{yu-lemma-6-10-1} that $\hat{v}_1$ is well defined.
\par
Let
\begin{equation*}\label{yu-7-5-bb-1}
	\kappa:=\frac{\sqrt{2}}{4e\hat{C}_{5}} \quad\text{with}\;\;\hat{C}_{5}=\sqrt{8\pi^2\sqrt{2+\lambda^{6}}}\;\;\text{given in Lemma
	\ref{yu-lemma-6-18-1}}.
\end{equation*}
We  define
$$V=V_1+V_2\quad\mbox{in}\;\; B_R(x_0)\times(-\kappa R,\kappa R)
$$
 with
 \begin{equation}\label{yu-6-23-6jia}
	V_1(x,y)=\frac{1}{2\pi}\int_{\mathbb{R}}e^{i\tau\mu}\hat{v}_1(x,\mu)
	\frac{\sinh(\sqrt{-i\mu}y)}{\sqrt{-i\mu}}d\mu\quad\mbox{in}\;B_R(x_0)\times(-\kappa R,\kappa R),
\end{equation}
\begin{equation}\label{yu-7-5-7}
	V_2(x,y)=\sum_{k=1}^\infty\alpha_ke^{-\mu_k\tau}f_k(x)\frac{\sinh(\sqrt{\mu_k}y)}
	{\sqrt{\mu_k}}\;\;\mbox{in}\;\;B_R(x_0)\times(-\kappa R,\kappa R)
\end{equation}
where $\{\mu_k\}_{k=1}^\infty$, $\{f_k\}_{k=1}^{\infty}$ are  given by $(\ref{yu-6-7-10})$, and
$\alpha_k=\langle u_2(\cdot,0),f_k\rangle_{L^2(B_R(x_0))}$.
Note from
Lemma \ref{yu-lemma-6-18-1} that $V_1$ is also well defined.
One can readily check that
\begin{equation}\label{yu-7-5-9}
\begin{cases}
	\mbox{div}(A(x)\nabla V(x,y))+V_{yy}(x,y)=0&\mbox{in}\;\;
	B_{\frac{R}{2}}(x_0)\times(-\kappa R,\kappa R),\\
	V(x,0)=0&\mbox{in}\;\;B_{\frac{R}{2}}(x_0),\\
	V_y(x,0)=u(x,\tau)&\mbox{in}\;\;B_{\frac{R}{2}}(x_0).
\end{cases}
\end{equation}	
Let $r\in (0, \kappa R/2)$. By the three-ball inequality for elliptic operators (cf., e.g., \cite[Theorem 3.1]{MV}), there exist constants
$C_{13}=C_{13}(\lambda,N,r)>0$ and $\beta_{2}=\beta_{2}(\lambda,N,r)\in(0,1)$ such that
\begin{equation}\label{yu-7-4-10}
      \|V\|_{L^2(B_{\frac{7r}{4}}(x_0,0))}\leq C_{13}\|V\|^{\beta_{2}}_{L^2(B_{r}(x_0,0))}\|V\|_{L^2(B_{2r}(x_0,0))}^{1-\beta_{2}}.
\end{equation}
Since $V_y$ satisfies the first equation of (\ref{yu-7-5-9}),	by the interior estimate of elliptic equations
we find
\begin{eqnarray}\label{yu-7-4-11}
	\int_{B_{\frac{7r}{4}}(x_0,0)}|V|^2dxdy&\geq &C_{14}r^2\int_{B_{\frac{3r}{2}}(x_0,0)}(|\nabla V|^2+|V_y|^2)dxdy\nonumber
	\\&\geq&\frac{C_{14}r^2}{2}
	\left(\int_{B_{\frac{3r}{2}}(x_0,0)}|V_y|^2dxdy+\int_{B_{\frac{3r}{2}}(x_0,0)}|V_y|^2dxdy\right)\nonumber\\
	&\geq&C_{14}r^2\left(\int_{B_{\frac{3r}{2}}(x_0,0)}|V_y|^2dxdy+r^2\int_{B_{\frac{5r}{4}}(x_0,0)}(|\nabla V_y|^2+|V_{yy}|^2)dxdy\right)\nonumber\\
	&\geq& C_{14}r^3\left(\frac{1}{r}\int_{B_{\frac{5r}{4}}(x_0,0)}|V_y|^2dxdy+r\int_{B_{\frac{5r}{4}}(x_0,0)}|V_{yy}|^2dxdy\right)
\end{eqnarray}
with $C_{14}=C_{14}(\lambda,N)>0$.

As a simple corollary of \cite[Lemma 9.9, Page 315]{Brezis}, we have the following  trace theorem
\begin{equation*}\label{yu-7-4-12}
	\int_{B_{r}(x_0)}|f(x,0)|^2dx\leq C_{15}(N)\left(\frac{1}{r}\int_{B_{\frac{5r}{4}}(x_0,0)}|f|^2dxdy+r\int_{B_{\frac{5r}{4}}
(x_0,0)}|f_y|^2dxdy\right)
\end{equation*}
        for any $f\in H^1(B_{5r/4}(x_0,0))$. Hence, by (\ref{yu-7-5-9}) and (\ref{yu-7-4-11}) we have
\begin{eqnarray}\label{yu-7-4-13}
	C_{16}r^3\int_{B_{r}(x_0)}|u(x,\tau)|^2dx\leq \int_{B_{\frac{7r}{4}}(x_0,0)}|V|^2dxdy
\end{eqnarray}	
with $C_{16}=C_{16}(\lambda,N)>0$.

By Lemma \ref{yu-proposition-7-1-1},  we obtain that there is  $C_{17}=C_{17}(\rho,N,r,\lambda)>0$ and  $\beta_{3}=\beta_{3}(\rho,N,r,\lambda)\in(0,1)$ such that
\begin{equation}\label{yu-7-4-2}
	\|V\|_{L^2(B_r(x_0,0))}\leq C_{17}\|V\|_{H^1(B_{\frac{3r}{2}}(x_0,0))}^{\beta_{3}}\|u(\cdot,\tau)\|_{L^2(\omega)}^{1-\beta_{3}}.
\end{equation}
Again, by the interior estimate, there is a constant $C_{18}=C_{18}(N)>0$ such that
\begin{equation}\label{yu-7-4-3}
	\|V\|_{H^1(B_{\frac{3r}{2}}(x_0,0))}\leq C_{18}r^{-1} \|V\|_{L^2(B_{2r}(x_0,0))}.
\end{equation}
Hence, it follows from  (\ref{yu-7-4-2}) and (\ref{yu-7-4-3}) that
\begin{equation}\label{yu-7-4-1}
	\|V\|_{L^2(B_r(x_0,0))}\leq C_{19}\|V\|_{L^2(B_{2r}(x_0,0))}^{\beta_{3}}\|u(\cdot,\tau)\|_{L^2(\omega)}^{1-\beta_{3}}
\end{equation}
with $C_{19}=C_{19}(\rho,N,r,\lambda)>0$.
It follows from  (\ref{yu-7-4-13}), (\ref{yu-7-4-10}) and \eqref{yu-7-4-1} that
\begin{equation}\label{yu-7-5-1}
	\|u(\cdot,\tau)\|_{L^2(B_{r}(x_0))}\leq C_{20}\|u(\cdot,\tau)\|_{L^2(\omega)}^{(1-\beta_{3})\beta_{2}}
	\|V\|^{1-(1-\beta_{3})\beta_{2}}_{L^2(B_{2r}(x_0,0))}
\end{equation}
with $C_{20}=C_{20}(\rho,N,r,\lambda)>0$.
\par
\medskip

To finish the proof, it suffices to bound the term $\|V\|_{L^2(B_{2r}(x_0,0))}$. Recall that $V=V_1+V_2$, we will treat
$V_1$ and $V_2$ separately.

In fact, we derive from (\ref{yu-6-23-6jia})  that
	for each $x\in B_{2r}(x_0)\subset B_{R}(x_0)$ and $|y|<\kappa R/(4\sqrt{2})$,
\begin{eqnarray*}\label{yu-7-5-2}
	|V_1(x,y)|
	&=&\left|\frac{1}{2\pi}\int_{\mathbb{R}}e^{i\tau\mu}\hat{v}_1(x,\mu)\int_{-y}^ye^{\sqrt{-i\mu}s}dsd\mu\right|
	\leq \frac{1}{2\pi}\int_{\mathbb{R}}|\hat{v}_1(x,\mu)|\int_{-y}^y|e^{\sqrt{-i\mu}s}|dsd\mu
	\nonumber\\
	&\leq&\frac{\kappa R}{4\sqrt{2}\pi}\int_{\mathbb{R}}|\hat{v}_1(x,\mu)|e^{\frac{1}{4\sqrt{2}}\kappa\sqrt{|\mu|}R}d\mu\nonumber\\
	&\leq &\frac{\kappa R}{4\sqrt{2}\pi}\left(\int_{\mathbb{R}}|\hat{v}_1(x,\mu)|^2e^{\frac{1}{\sqrt{2}}\kappa\sqrt{|\mu|}R}d\mu\right)^{\frac{1}{2}}
	\left(\int_{\mathbb{R}}e^{-\frac{1}{2\sqrt{2}}\kappa \sqrt{|\mu|}R}d\mu\right)^{\frac{1}{2}}\nonumber\\
	&=&\frac{1}{2\pi}\left(\int_{\mathbb{R}}|\hat{v}_1(x,\mu)|^2e^{\frac{1}{\sqrt{2}}\kappa\sqrt{|\mu|}R}d\mu\right)^{\frac{1}{2}}.
\end{eqnarray*}
	Hence, by Lemma \ref{yu-lemma-6-18-1}, we have for each $r<\kappa R/(4\sqrt{2})$,
\begin{eqnarray}\label{yu-7-5-3}
	\int_{B_{2r}(x_0,0)}|V_1|^2dxdy
	&\leq& \frac{1}{2\pi^2}\hat{C}_{4}^2Re^{\frac{2\hat{C}_{1}T}{T-\tau}}
	\left[T+R^{2}\right]^{2}F^2(R)
	\int_{\mathbb{R}}e^{-\frac{1}{\sqrt{2}}\kappa \sqrt{|\mu|}R}d\mu\nonumber\\
	&\leq&\frac{32e}{\pi^2}\hat{C}_{5}^2\hat{C}_{4}^2R^{-1}e^{\frac{2\hat{C}_{1}T}{T-\tau}}
	\left[T+R^{2}\right]^{2}F^2(R).
\end{eqnarray}
While, by  (\ref{yu-7-5-7}) and (\ref{yu-7-4-7}) we obtain
\begin{eqnarray}\label{yu-7-5-10}
	\int_{B_{2r}(x_0,0)}|V_2|^2dxdy&\leq&\int_{-2r}^{2r}\int_{B_{R}(x_0)}|V_2|^2dxdy
	\leq \int_{-2r}^{2r}\sum_{k=1}^\infty\alpha_k^2e^{-2\mu_k\tau}\left|\frac{\sinh(\sqrt{\mu_k}y)}{\sqrt{\mu_k}}\right|^2dy\nonumber\\
	&\leq&\frac{e^{\frac{2r^{2}}{\tau}}}{4\hat{C}_{2}^{3/2}}R^{3}\sum_{k=1}^\infty\alpha_k^2=
	\frac{e^{\frac{2r^{2}}{\tau}}}{4\hat{C}_{2}^{3/2}}R^{3}\int_{B_R(x_0)}|u(x,0)|^2dx\nonumber\\
	&\leq& C_{21}R^{3}e^{\frac{C_{21}}{\tau}}F^2(R)
\end{eqnarray}
with $C_{21}=C_{21}(N,r,\lambda)>0$.
	Therefore, by (\ref{yu-7-5-3}) and (\ref{yu-7-5-10}), we conclude that
\begin{equation*}\label{yu-7-5-11}
	\|V\|_{L^2(B_{2r}(x_0,0))}\leq C_{22}\left[e^{\frac{\hat{C}_{1}T}{T-\tau}}R^{-1}
	\left(T+R^{2}\right)^{2}+e^{\frac{C_{21}}{\tau}}R^{3}\right]^{\frac{1}{2}}F(R)
\end{equation*}
with $C_{22}=C_{22}(N,r,\lambda)>0$.
	This, together with (\ref{yu-7-5-1}), means that
\begin{eqnarray*}\label{yu-7-5-12}
	\|u(\cdot,\tau)\|_{L^2(B_{r}(x_0))}
	&\leq& C_{20}\|u(\cdot,\tau)\|_{L^2(\omega)}^{(1-\beta_{3})\beta_{2}}
	\|V\|^{1-(1-\beta_{3})\beta_{2}}_{L^2(B_{2r}(x_0,0))}\nonumber\\
&\leq& C_{20}C_{22}^{1-(1-\beta_{3})\beta_{2}}\left[e^{\frac{\hat{C}_{1}T}{T-\tau}}R^{-1}
	\left(T+R^{2}\right)^{2}+e^{\frac{C_{21}}{\tau}}R^{3}\right]^{\frac{1-(1-\beta_{3})\beta_{2}}{2}}\nonumber\\
&&\times \|u(\cdot,\tau)\|_{L^2(\omega)}^{(1-\beta_{3})\beta_{2}}F(R)^{1-(1-\beta_{3})\beta_{2}}.
\end{eqnarray*}
	Taking $\sigma=(1-\beta_{3})\beta_{2}$, the proof is immediately achieved by the arbitrariness of $u_0$.
\end{proof}

\begin{lemma}\label{yu-LEMMA-7-10-6}
Let $T>0$, $0<r<+\infty$ and $\rho>0$. Assume that there is a sequence $\{x_i\}_{i\in\mathbb{N}^+}\subset\mathbb R^N$ so that
\begin{equation*}\mathbb{R}^{N}=\bigcup_{i\in\mathbb{N}^+}Q_{r}(x_{i})
\quad \text{with}\quad \mathrm{int}(Q_{r}(x_{i}))\bigcap \mathrm{int}(Q_{r}(x_{j}))=\emptyset\quad \text{for each}\quad i\neq j\in\mathbb N^+.
\end{equation*}
Let $$\widetilde{\omega}\triangleq\bigcup_{i\in\mathbb{N}^+}\widetilde{\omega}_{i} \quad\text{with} \quad \widetilde{\omega}_{i} \subset B_{r/2}(x_{i})\quad\text{satisfy}\quad \frac{|\widetilde{\omega}_{i}|}{|B_{r}(x_i)|}\geq\rho \quad\text{for each}\quad i\in\mathbb N^+ $$
where, for each $i\in\mathbb{N}^+$, $\widetilde{\omega}_{i}$ is a $N$-dimensional Lebesgue measurable set of positive measure.
There exist constants $C=C(r,N,\lambda,\rho)>0$ and  $\sigma=\sigma(r,N,\lambda,\rho)\in(0,1)$ such that for any $u_{0}\in H^{1}(\mathbb R^N)$, the corresponding solution $u$ of \eqref{yu-6-24-1} satisfies
\begin{equation*}\label{yu-7-10-2}
	\|u(\cdot,T)\|_{L^2(\mathbb R^N)}\leq C
	\left(T^{3}+e^{\frac{C}{T}}\right)\|u(\cdot,T)\|_{L^2(\tilde{\omega})}^\sigma
	\left(\sup_{s\in[0,2T]}\|u(\cdot,s)\|_{H^1(\mathbb R^N)}\right)^{1-\sigma}.
	\end{equation*}
\end{lemma}
\begin{proof}
By Lemma~\ref{lemma-2A2} (where $r$ and $\omega$ are replaced by  $ \sqrt{N}r$ and $\widetilde{\omega}_{i}$ ($i\in\mathbb{N}^+$), respectively), we get that
\begin{eqnarray*}
 \int_{Q_{r}(x_{i})}|u(x,\tau)|^{2}\mathrm{d}x
 &\leq& \int_{B_{\sqrt{N}r}(x_{i})}|u(x,\tau)|^{2}\mathrm{d}x\\
 &\leq&\hat{C}_{9}\left(T^{2}e^{\frac{\hat{C}_{1}T}{T-\tau}}
 +e^{\frac{\hat{C}_{9}}{\tau}}\right)^{2}\left(\int_{\widetilde{\omega}_{i}}
 |u(x,\tau)|^{2}\mathrm{d}x\right)^\theta\\
 &&\times\left(\sup_{s\in[0,T]}\|u(\cdot,s)\|^{2}_{H^1(B_{\sqrt{N}\varrho r}(x_i))}\right)^{1-\theta},
 \end{eqnarray*}
where $\hat{C}_{9}=\hat{C}_{9}(r,N,\lambda,\rho)>0$ and $\theta=\theta(r,N,\lambda,\rho)\in(0,1)$. This, along with Young's  inequality,
implies that for each $\varepsilon>0,$
\begin{eqnarray*}
 \int_{Q_{r}(x_{i})}|u(x,\tau)|^{2}\mathrm{d}x
 &\leq&\hat{C}^{2}_{9}
 \biggl(T^{2}e^{\frac{\hat{C}_{1}T}{T-\tau}}+e^{\frac{\hat{C}_{9}}{\tau}}\biggl)^{2}
 \biggl(\varepsilon\theta\sup_{s\in[0,T]}\|u(\cdot,s)\|^{2}_{H^1(B_{\sqrt{N}\varrho r}(x_i))}\\
 &&\;\;\;\;\;\;\;\;+\varepsilon^{-\frac{\theta}{1-\theta}}(1-\theta)
 \int_{\widetilde{\omega}_{i}}|u(x,\tau)|^{2}\mathrm{d}x\biggl).
 \end{eqnarray*}
Then
\begin{equation}\label{3.44444}
\begin{array}{lll}
 &&\displaystyle{}\int_{\mathbb{R}^{N}}|u(x,\tau)|^{2}\mathrm{d}x
 =\sum_{i\in\mathbb{N}^+}\int_{Q_{r}(x_{i})}|u(x,\tau)|^{2}\mathrm{d}x\\
 &\leq&\hat{C}^{2}_{9}\left(T^{2}e^{\frac{\hat{C}_{1}T}{T-\tau}}+e^{\frac{\hat{C}_{9}}{\tau}}\right)^{2}
 \left(\varepsilon\theta \displaystyle\sum_{i\in\mathbb{N}^+}\sup_{s\in[0,T]}\|u(\cdot,s)\|^{2}_{H^1(B_{\sqrt{N}\varrho r}(x_i))}+\varepsilon^{-\frac{\theta}{1-\theta}}(1-\theta)
 \int_{\widetilde{\omega}}|u(x,\tau)|^{2}\mathrm{d}x\right).
\end{array}
\end{equation}
\par
 Next, let $R:=\varrho r$, we show that
\begin{equation}\label{yu-4-9-4}
\sum_{i\in\mathbb{N}^+}\sup_{s\in[0,T]}\|u(\cdot,s)\|^{2}_{H^1(B_{\sqrt{N}R}(x_i))}\leq \hat{C}_{10}(1+T)\sup_{s\in[0,T]}\|u(\cdot,s)\|^2_{H^1(\mathbb{R}^{N})},
 \end{equation}
 where $\hat{C}_{10}(\lambda,R,N)>0$. Indeed, we take a cut-off function $\psi\in C_0^\infty(\mathbb{R}^N;[0,1])$ so that $\mbox{supp}\psi\Subset Q_{2\sqrt{N}R}(0)$ and $\psi\equiv 1$ in
 $B_{\sqrt{N}R}(0)$. For each $i\in\mathbb{N}^+$, we let $\psi_i(x):=\psi(x-x_i)$ for any
 $x\in\mathbb{R}^+$. It is clear that $\mbox{supp}\psi_i\Subset Q_{2\sqrt{N}R}(x_i)$ and
 $\psi_i\equiv 1$ in $B_{\sqrt{N}R}(x_i)$. We suppose that
$$
\begin{cases}
    \mbox{Card}\{j\in\mathbb{N}^+: Q_{r}(x_1)\cap  Q_{2\sqrt{N}R}(x_j)\neq \emptyset\}=m^*,\\
    \mbox{Card}\{j\in\mathbb{N}^+: Q_{2\sqrt{N}R}(x_1)\cap  Q_{r}(x_j)\neq \emptyset\}=n^*.
\end{cases}
$$
    This means that, for each $k\in\mathbb{N}^+$,
$$
\begin{cases}
    \mbox{Card}\{j\in\mathbb{N}^+: Q_{r}(x_k)\cap  Q_{2\sqrt{N}R}(x_j)\neq \emptyset\}=m^*,\\
    \mbox{Card}\{j\in\mathbb{N}^+: Q_{2\sqrt{N}R}(x_k)\cap  Q_{r}(x_j)\neq \emptyset\}=n^*.
\end{cases}
$$
    For each $k$, we suppose $\{y_{k,i}\}_{i=1}^{n^*}
    =\{x_i|Q_{2\sqrt{N}R}(x_k)\cap Q_r(x_i)\neq \emptyset\}$.
    Thus, for any $f\in H^1(\mathbb{R}^N)$,
\begin{eqnarray}\label{yu-4-9-1}
    \sum_{k\in\mathbb{N}^+}\|f\|^{2}_{H^1(Q_{2\sqrt{N}R}(x_k))}
    &\leq&\sum_{k\in\mathbb{N}^+}\sum_{i=1}^{n^*}\|f\|^{2}_{H^1(Q_r(y_{k,i})}
    \leq \sum_{i=1}^{n^*}\sum_{k\in\mathbb{N}^+}\|f\|^2_{H^1(Q_r(y_{k,i}))}\nonumber\\
    &\leq&m^*n^*\sum_{k\in\mathbb{N}^+}\|f\|^2_{H^1(Q_r(x_k))}=m^*n^*\|f\|^2_{H^1(\mathbb{R}^N)}.
\end{eqnarray}
    For each $i\in\mathbb{N}^+$, we let $v_i=\psi_iu$. It is obvious that $v_i$ verifies
\begin{equation}\label{yu-4-9-2}
\begin{cases}
    (v_i)_t-\mbox{div}(A(x)\nabla v_i)=-2\nabla\psi_i\cdot (A(x)\nabla u)-u\mbox{div}(A(x)\nabla \psi_i)
    &\mbox{in}\;\;\; Q_{2\sqrt{N}R}(x_i)\times [0,T],\\
    v_i=0&\mbox{on}\;\; \partial Q_{2\sqrt{N}R}(x_i)\times [0,T],\\
    v_i(0)=\psi_iu_0 &\mbox{in}\;\;\;Q_{2\sqrt{N}R}(x_i).
\end{cases}
\end{equation}
    By the standard energy estimate for (\ref{yu-4-9-2}), one can easily check that
\begin{eqnarray*}
    &\;&\sup_{s\in[0,T]}\|u(\cdot,s)\|_{H^1(B_{\sqrt{N}R}(x_i))}^2
    \leq \sup_{s\in[0,T]}\|v_i(\cdot,s)\|_{H^1(Q_{2\sqrt{N}R}(x_i))}^2\nonumber\\
    &\leq& C(\lambda,N)(1+R^{-4})\left(\|u_0\|_{H^1(Q_{2\sqrt{N}R}(x_i))}^2+\int_0^T
    \|u(\cdot, s)\|_{H^1(Q_{2\sqrt{N}R}(x_i))}^2ds\right),
\end{eqnarray*}
    where $C(\lambda,N)>0$. This, along with (\ref{yu-4-9-1}), yields that
\begin{eqnarray*}
    &&\sum_{i\in\mathbb{N}^+}\sup_{s\in[0,T]}\|u(\cdot,s)\|_{H^1(B_{\sqrt{N}R}(x_i))}^2\\
    &\leq& C(\lambda,N)(1+R^{-4})\sum_{i\in\mathbb{N}^+}\left(\|u_0\|_{H^1(Q_{2\sqrt{N}R}(x_i))}^2+\int_0^T
    \|u(\cdot, s)\|_{H^1(Q_{2\sqrt{N}R}(x_i))}^2ds\right)\nonumber\\
    &\leq&m^*n^*C(\lambda,N)(1+R^{-4})(1+T)\sup_{s\in[0,T]}\|u(\cdot,s)\|^2_{H^1(\mathbb{R}^N)}.
\end{eqnarray*}
    Thus, (\ref{yu-4-9-4}) is true.
   By (\ref{yu-4-9-4}) and (\ref{3.44444}), we obtain that
 \begin{eqnarray*}
\int_{\mathbb{R}^{N}}|u(x,\tau)|^{2}\mathrm{d}x
 &\leq&\displaystyle{}\hat{C}^{2}_{9}\displaystyle{}\left(T^{2}
 e^{\frac{\hat{C}_{1}T}{T-\tau}}+e^{\frac{\hat{C}_{9}}{\tau}}\right)^{2}
 \biggl(\varepsilon\theta \hat{C}_{10}(1+T)\sup_{s\in[0,T]}\|u(\cdot,s)\|^2_{H^1(\mathbb{R}^{N})}
 \\
 &&\;\;\;\;\;\;\;\;\;+\varepsilon^{-\frac{\theta}{1-\theta}}(1-\theta)
 \int_{\widetilde{\omega}}|u(x,\tau)|^{2}\mathrm{d}x\biggl)\\
 \end{eqnarray*}
for each $ \varepsilon>0.$
 This implies that
 \begin{equation*}\label{3.444440}
\begin{array}{lll}
 \displaystyle{}\int_{\mathbb{R}^{N}}|u(x,\tau)|^{2}\mathrm{d}x
 &\leq&\displaystyle{}\hat{C}^{2}_{9}\displaystyle{} \left(T^{2}e^{\frac{\hat{C}_{1}T}{T-\tau}}+e^{\frac{\hat{C}_{9}}{\tau}}\right)^{2}
 \left(\hat{C}_{10}(1+T)\sup_{s\in[0,T]}\|u(\cdot,s)\|^2_{H^1(\mathbb{R}^{N})}\right)^\theta\\
 &&\times\left(\int_{\widetilde{\omega}}
 |u(x,\tau)|^{2}\mathrm{d}x\right)^{1-\theta}.
\end{array}
\end{equation*}
 By Lemma~\ref{lemma-2A2} (where $\tau$ and $T$ are replaced by  $T$ and $2T$, respectively),  we finish the proof of  Lemma \ref{yu-LEMMA-7-10-6}.
\end{proof}
\vskip 5pt
\noindent\textbf{Proof of Theorem \ref{yu-theorem-7-10-6}.}
First, by standard energy estimates of solutions to \eqref{yu-6-24-1} we have
\begin{equation}\label{yu-7-12-2}
	\|u(\cdot,t)\|_{H^1(\mathbb R^N)}\leq \frac{C_{23}e^{C_{23}t}}{\sqrt{t}}\|u_0\|_{L^2(\mathbb R^N)}
	\end{equation}
with $C_{23}=C_{23}(N,\lambda)>0$, for each $t\in(0,6T]$. Moreover, if $u_0\in H^1(\mathbb R^N)$, then
 \begin{equation}\label{yu-7-12-3}
 	\|u(\cdot,t)\|_{H^1(\mathbb R^N)}\leq C_{24}e^{C_{24}t}\|u_0\|_{H^1(\mathbb R^N)}\;\;\mbox{for each}\;\;
	t\in[0,6T],
 \end{equation}
where $C_{24}=C_{24}(N,\lambda)>0$.

Second, we consider the following equation
\begin{equation*}\label{yu-7-12-12}
\begin{cases}
	v_t-\mbox{div}(A(x)\nabla v)=0&\mbox{in}\;\;\mathbb R^N\times(0,4T),\\
	v(\cdot,0)=u(\cdot,\frac{T}{2})&\mbox{in}\;\;\mathbb R^N.
\end{cases}
\end{equation*}
	It is obvious that $v(\cdot,t)=u(\cdot,t+T/2)$ when $t\in[0,4T]$. Moreover, by (\ref{yu-7-12-2}) we have
	$u(\cdot,T/2)\in H^1(\mathbb R^N)$,
	which means that $v\in C([0,4T];H^1(\mathbb R^N))$.
	From Lemma \ref{yu-LEMMA-7-10-6} (where $T$ and $\tilde{\omega}$ are replaced by $T/2$ and $\omega$, respectively), it follows that there are $C>0$ and $\sigma\in(0,1)$ such that
\begin{equation*}\label{yu-7-13-1}
	\left\|v\left(\cdot,\frac{T}{2}\right)\right\|_{L^2(\mathbb R^N)}\leq C\left(T^{3}+e^{\frac{C}{T}}\right) \left\|v\left(\cdot,\frac{T}{2}\right)\right\|^\sigma_{L^2(\omega)}\left(\sup_{s\in[0,T]}\|v(\cdot,s)\|_{H^1(\mathbb R^N)}\right)^{1-\sigma}.
\end{equation*}
	This, along with (\ref{yu-7-12-3}), gives that
\begin{equation*}\label{yu-7-13-2}
	\left\|v\left(\cdot,\frac{T}{2}\right)\right\|_{L^2(\mathbb R^N)}\leq C C_{24}e^{C_{24}T}\left(T^{3}+e^{\frac{C}{T}}\right) \left\|v\left(\cdot,\frac{T}{2}\right)\right\|^\sigma_{L^2(\omega)}\|v(\cdot,0)\|^{1-\sigma}_{H^1(\mathbb R^N)}. \end{equation*}
Which is
\begin{equation*}\label{yu-7-13-3}
	\left\|u\left(\cdot,T\right)\right\|_{L^2(\mathbb R^N)}\leq  C C_{24}e^{C_{24}T}\left(T^{3}+e^{\frac{C}{T}}\right) \left\|u\left(\cdot,T\right)\right\|^\sigma_{L^2(\omega)}\left\|u\left(\cdot,\frac{T}{2}\right)\right\|^{1-\sigma}_{H^1(\mathbb R^N)}.
\end{equation*}
	This, together with (\ref{yu-7-12-2}), completes the proof.\qed

\section{Proof of Theorem \ref{jiudu4}}\label{finalproof}

Now, we are able to present the proof of  Theorem \ref{jiudu4}.\\

\noindent\textbf{Proof of Theorem \ref{jiudu4}.}
By Theorem \ref{yu-theorem-7-10-6} (where $r, x_{i}$ and $w_{i}$ are replaced by
$r, x_{i}$ and $w_{i}$, respectively) and Young's inequality,
for any $0\leq t_{1}<t_{2}\leq T$, we see that
 \begin{equation}\label{2019-7-9}
 \|u(t_{2})\|^{2}_{L^{2}(\mathbb{R}^{N})}\leq\varepsilon
 \|u(t_{1})\|^{2}_{L^{2}(\mathbb{R}^{N})}+
 \frac{C_{25}}{\varepsilon^{\alpha}}e^{\frac{C_{26}}{t_{2}-t_{1}}}
 \|u(t_{2})\|^{2}_{L^{2}(\omega)} \ \ \ \mathrm{for\ each}\ \varepsilon>0,
 \end{equation}
where $C_{25}:=\left(Ce^{CT}\right)^{2/(1-\sigma)}$,
$C_{26}:=2C/(1-\sigma)$
and $\alpha:= \sigma/(1-\sigma)$.
Let $l$ be a density point of $E$. According to Proposition 2.1 in \cite{Phung-Wang-2013},
for each $\kappa>1$, there exists $l_{1}\in (l,T)$, depending on $\kappa$ and $E$,
so that the sequence $\{l_{m}\}_{m\in\mathbb{N}^+}$, given by
$$
l_{m+1}=l+\frac{1}{\kappa^{m}}(l_{1}-l)\;\;\mbox{for each}\;\;m\in\mathbb{N}^+,
$$
satisfies that
 \begin{equation}\label{3.2525251}
l_{m}-l_{m+1}\leq 3|E\cap(l_{m+1},l_{m})|.
 \end{equation}
\par
Next, let $0<l_{m+2}<l_{m+1}\leq t<l_{m}<l_{1}<T$. It follows from (\ref{2019-7-9}) that
\begin{equation}\label{3.2525252}
\|u(t)\|^{2}_{L^{2}(\mathbb{R}^{N})}\leq \varepsilon\|u(l_{m+2})\|^{2}_{L^{2}(\mathbb{R}^{N})}
+\frac{C_{25}}{\varepsilon^{\alpha}}
e^{\frac{C_{26}}{t-l_{m+2}}}\|u(t)\|^{2}_{L^{2}(\omega)} \ \mathrm{for\ each}\ \varepsilon>0.
 \end{equation}
By a standard energy estimate, we have that
$$
\|u(l_{m})\|_{L^{2}(\mathbb{R}^{N})}\leq C_{27}\|u(t)\|_{L^{2}(\mathbb{R}^{N})},
$$
where $C_{27}=C_{27}(\lambda,N)\geq1.$
This, along with (\ref{3.2525252}), implies that
$$
\|u(l_{m})\|^{2}_{L^{2}(\mathbb{R}^{N})}\leq C^{2}_{27}\left(\varepsilon\|u(l_{m+2})
\|^{2}_{L^{2}(\mathbb{R}^{N})}+
\frac{C_{25}}{\varepsilon^{\alpha}}
e^{\frac{C_{26}}{t-l_{m+2}}}\|u(t)\|^{2}_{L^{2}(\omega)}\right)
\ \mathrm{for\ each}\ \varepsilon>0,
$$
which indicates that
$$\|u(l_{m})\|^{2}_{L^{2}(\mathbb{R}^{N})}
\leq \varepsilon\|u(l_{m+2})\|^{2}_{L^{2}(\mathbb{R}^{N})}
+\frac{C_{28}}{\varepsilon^{\alpha}}
e^{\frac{C_{26}}{t-l_{m+2}}}\|u(t)\|^{2}_{L^{2}(\omega)}
 \ \mathrm{for\ each}\ \varepsilon>0,
 $$
where $C_{28}=C_{27}^{2(1+\alpha)}C_{25}$.
Integrating the latter inequality over $ E\cap(l_{m+1},l_{m})$, we get that
\begin{equation}\label{3.2525253}
\begin{array}{lll}
 \displaystyle{}|E\cap(l_{m+1},l_{m})|\|u(l_{m})\|^{2}_{L^{2}(\mathbb{R}^{N})}
 &\leq&\displaystyle{}\varepsilon |E\cap(l_{m+1},l_{m})|\|u(l_{m+2})\|^{2}_{L^{2}(\mathbb{R}^{N})}\\
 &&\displaystyle{}+\frac{C_{28}}
 {\varepsilon^{\alpha}}e^{\frac{C_{26}}{l_{m+1}-l_{m+2}}}
 \int_{l_{m+1}}^{l_{m}}\chi_{E}\|u(t)\|^{2}_{L^{2}(\omega)}\mathrm{d}t
 \ \mathrm{for\ each}\ \varepsilon>0.
\end{array}
\end{equation}
Here and throughout the proof of Theorem~\ref{jiudu4}, $\chi_{E}$ denotes the characteristic function of $E$. Since $l_{m}-l_{m+1}=(\kappa-1)(l_{1}-l)/\kappa^{m},$ by (\ref{3.2525253}) and (\ref{3.2525251}), we obtain that
\begin{eqnarray*}
\|u(l_{m})\|^{2}_{L^{2}(\mathbb{R}^{N})}&\leq& \varepsilon \|u(l_{m+2})\|^{2}_{L^{2}(\mathbb{R}^{N})}+\frac{1}{|E\cap(l_{m+1},l_{m})|}
\frac{C_{28}}{\varepsilon^{\alpha}}
e^{\frac{C_{26}}{l_{m+1}-l_{m+2}}}
\int_{l_{m+1}}^{l_{m}}\chi_{E}\|u(t)\|^{2}_{L^{2}(\omega)}\mathrm{d}t\\
&\leq&\frac{3\kappa^{m}}{(l_{1}-l)(\kappa-1)}
\frac{C_{28}}{\varepsilon^{\alpha}}
e^{C_{26}\left(\frac{1}{l_{1}-l}\frac{\kappa^{m+1}}{\kappa-1}\right)}
\int_{l_{m+1}}^{l_{m}}\chi_{E}\|u(t)\|^{2}_{L^{2}(\omega)}\mathrm{d}t+
\varepsilon \|u(l_{m+2})\|^{2}_{L^{2}(\mathbb{R}^{N})}
 \end{eqnarray*}
 for each $\varepsilon>0$.
This yields that
\begin{equation}\label{3.2525254}
\begin{array}{lll}
 \displaystyle{}\|u(l_{m})\|^{2}_{L^{2}(\mathbb{R}^{N})}&\leq& \displaystyle{}
 \frac{1}{\varepsilon^{\alpha}}\frac{3}{\kappa}
 \frac{C_{28}}{C_{26}}
 e^{2C_{26}\left(\frac{1}{l_{1}-l}
 \frac{\kappa^{m+1}}{\kappa-1}\right)}\int_{l_{m+1}}^{l_{m}}\chi_{E}
 \|u(t)\|^{2}_{L^{2}(\omega)}\mathrm{d}t\displaystyle{}+
 \varepsilon \|u(l_{m+2})\|^{2}_{L^{2}(\mathbb{R}^{N})}
\end{array}
\end{equation}
for each $\varepsilon>0$.
Denote $d:= 2C_{26}/[\kappa(l_{1}-l)(\kappa-1)]$.
It follows from (\ref{3.2525254}) that
\begin{eqnarray*}
\varepsilon^{\alpha}e^{-d\kappa^{m+2}}\|u(l_{m})\|^{2}_{L^{2}(\mathbb{R}^{N})}
-\varepsilon^{1+\alpha}e^{-d\kappa^{m+2}}\|u(l_{m+2})\|^{2}_{L^{2}(\mathbb{R}^{N})}
\leq\frac{3}{\kappa}\frac{C_{28}}
{C_{26}}\int_{l_{m+1}}^{l_{m}}\chi_{E}\|u(t)\|^{2}_{L^{2}(\omega)}\mathrm{d}t
\end{eqnarray*}
for each $\varepsilon>0$.
Choosing $\varepsilon=e^{-d\kappa^{m+2}}$ in the latter inequality, we observe that
\begin{equation}\label{3.25252555}
\begin{array}{lll}
 \displaystyle{}e^{-(1+\alpha)d\kappa^{m+2}}\|u(l_{m})\|^{2}_{L^{2}(\mathbb{R}^{N})}
 -e^{-(2+\alpha)d\kappa^{m+2}}\|u(l_{m+2})\|^{2}_{L^{2}(\mathbb{R}^{N})}
 \leq\displaystyle{}\frac{3}{\kappa}\frac{C_{28}}
 {C_{26}}\int_{l_{m+1}}^{l_{m}}\chi_{E}\|u(t)\|^{2}_{L^{2}(\omega)}\mathrm{d}t.
\end{array}
\end{equation}
Take $\kappa=\sqrt{(\alpha+2)/(\alpha+1)}$ in (\ref{3.25252555}). Then we have that
\begin{eqnarray*}
e^{-(2+\alpha)d\kappa^{m}}\|u(l_{m})\|^{2}_{L^{2}(\mathbb{R}^{N})}
-e^{-(2+\alpha)d\kappa^{m+2}}\|u(l_{m+2})\|^{2}_{L^{2}(\mathbb{R}^{N})}
\leq \frac{3}{\kappa}\frac{C_{28}}{C_{26}}
\int_{l_{m+1}}^{l_{m}}\chi_{E}\|u(t)\|^{2}_{L^{2}(\omega)}\mathrm{d}t.
\end{eqnarray*}
Changing $m$ to $2m'$ and summing the above inequality from $m'=1$ to infinity give the desired result. Indeed,
\begin{eqnarray*}
&&C_{27}^{-2}e^{-(2+\alpha)d\kappa^{2}}\|u(T)\|^{2}_{L^{2}(\mathbb{R}^{N})}
\leq e^{-(2+\alpha)d\kappa^{2}}\|u(l_{2})\|^{2}_{L^{2}(\mathbb{R}^{N})}\\
&\leq&\sum_{m'=1}^{+\infty}\left(e^{-(2+\alpha)d\kappa^{2m'}}\|u(l_{2m'})\|_{L^{2}(\mathbb{R}^{N})}
-e^{-(2+\alpha)d\kappa^{2m'+2}}\|u(l_{2m'+2})\|^{2}_{L^{2}(\mathbb{R}^{N})}\right)\\
&\leq& \frac{3}{\kappa}\frac{C_{28}}{C_{26}}\sum_{m'=1}^{+\infty}
\int_{l_{2m'+1}}^{l_{2m'}}\chi_{E}\|u(t)\|^{2}_{L^{2}(\omega)}\mathrm{d}t
\leq \frac{3}{\kappa}\frac{C_{28}}{C_{26}}\int_{0}^{T}\chi_{E}\|u(t)\|^{2}_{L^{2}(\omega)}\mathrm{d}t.
 \end{eqnarray*}

In summary, we finish the proof of Theorem~\ref{jiudu4}.\qed

\section{Appendix: Proof of (\ref{yu-6-22-4})}

Let $m\in\mathbb{N}^+$ and $a_j=1-j/(2m)$ for $j=0,1,\dots,m+1$. For each $j\in\{0,1,\cdots, m\}$, we define a cutoff function
\begin{equation*}\label{yu-6-19-4}
	h_j(s):=
\begin{cases}
	0&\mbox{if}\;\;|s|>a_j,\\
	\frac{1}{2}\left[1+\cos\left(\frac{\pi(a_{j+1}-|s|)}{a_{j+1}-a_j}\right)\right]
	&\mbox{if}\;\;a_{j+1}\leq |s|\leq a_j,\\
	1&\mbox{if}\;\;|s|<a_{j+1}.
\end{cases}
\end{equation*}
Clearly,
\begin{equation*}\label{yu-6-20-2}
	|h_j'(s)|\leq m\pi \;\;\mbox{for any}\;\;s\in\mathbb{R}.
\end{equation*}
Denote
$p_j=\partial^jp/\partial\xi^j$, $j=0,1,\dots,m$. Then $p_j$ verifies
\begin{equation}\label{yu-6-19-6}
	\mbox{div}(A\nabla p_j(\cdot,\cdot,\mu))+i\mbox{sign}(\mu)p_{j+2}(\cdot,\cdot,\mu)
	=0\;\;\mbox{in}\;\;B_R(x_0)\times\mathbb{R}.
\end{equation}	
	Let
\begin{equation*}\label{yu-6-19-7}
	\eta_j(x,\xi)=h_j\left(\frac{|x-x_0|}{R}\right)h_j\left(\frac{\xi}{R}\right)\quad\text{for}\;\;(x,\xi)\in B_R(x_0)\times\mathbb{R}.
\end{equation*}
	Multiplying first (\ref{yu-6-19-6}) by
	$\bar{p}_j\eta_j^2$ and then integrating by parts over $D_j=B_{a_jR}(x_0)\times(-a_jR,a_jR)$, we obtain
\begin{eqnarray}\label{yu-6-19-8}
	&\;&-\int_{D_j}\nabla\bar{p}_j\cdot(A\nabla p_j)\eta_j^2dxd\xi
	-i\mbox{sign}(\mu)\int_{D_j}|p_{j+1}|^2\eta_j^2dxd\xi\nonumber\\
	&=&\int_{D_j}\nabla\eta_j^2\cdot(A\nabla p_i)\bar{p}_jdxd\xi+i\mbox{sign}(\mu)\int_{D_j}p_{i+1}\partial_\xi\eta^2_j\bar{p}_jdxd\xi.
\end{eqnarray}
Since $\nabla\bar{p}_j\cdot(A\nabla p_j)$ and $|p_{j+1}|^2$ are real-valued,
 by (\ref{yu-6-19-8}) we get
\begin{eqnarray}\label{yu-6-19-9}
	&\;&\left(\int_{D_j}\nabla\bar{p}_j\cdot(A\nabla p_j)\eta_j^2dxd\xi\right)^2
	+|\mbox{sign}(\mu)|^2\left(\int_{D_j}|p_{j+1}|^2\eta_j^2dxd\xi\right)^2\nonumber\\
	&\leq&2\left(\int_{D_j}|\nabla \eta_j^2\cdot(A\nabla p_j)||p_j|dxd\xi\right)^2
+2\left(\int_{D_j}|p_{j+1}||p_j||\partial_{\xi}\eta^2_j|dxd\xi\right)^2
	:=2\sum_{i=1}^2I_i.
\end{eqnarray}
\par
	Next, we will estimate $I_i$ $(i=1,2)$ one by one. By Young's inequality, we have
\begin{eqnarray*}\label{yu-6-21-1}
	&\;&\int_{D_j}|\nabla \eta_j^2\cdot(A\nabla p_j)||p_j|dxd\xi
	\leq 2\lambda\int_{D_j}|\nabla \eta_j||\eta_j||\nabla p_j||p_j|dxd\xi\nonumber\\
	&\leq&\epsilon_1\lambda\int_{D_j}|\nabla p_j|^2\eta_j^2dxd\xi
	+\frac{\lambda}{\epsilon_1}\int_{D_j}|p_j|^2|\nabla\eta_j|^2dxd\xi\nonumber\\
	&\leq&\epsilon_1\lambda\int_{D_j}|\nabla p_j|^2\eta_j^2dxd\xi
	+\frac{\lambda\pi^2m^2}{R^2\epsilon_1}\int_{D_j}|p_j|^2dxd\xi.
\end{eqnarray*}
Thus, we get that
\begin{equation}\label{yu-6-21-2}
	I_1\leq 2\lambda^2\epsilon_1^2\left(\int_{D_j}|\nabla p_j|^2\eta_j^2dxd\xi\right)^2
	+\frac{2\lambda^2\pi^4m^4}{R^4\epsilon^2_1}\left(\int_{D_j}|p_j|^2dxd\xi\right)^2, \;\;\forall\ \epsilon_1>0.
\end{equation}
Furthermore,
\begin{eqnarray*}\label{yu-6-21-3}
	\int_{D_j}|p_{j+1}|p_j||\partial_{\xi}\eta^2_j|dxd\xi
	&\leq&\epsilon_2\int_{D_j}|p_{j+1}|^2\eta_j^2dxd\xi
	+\frac{1}{\epsilon_2}\int_{D_j}|p_j|^2|\partial_\xi\eta_j|^2dxd\xi\nonumber\\
	&\leq&\epsilon_2\int_{D_j}|p_{j+1}|^2\eta_j^2dxd\xi
	+\frac{m^2\pi^2}{R^2\epsilon_2}\int_{D_j}|p_j|^2dxd\xi.
\end{eqnarray*}	
Thus, we have that
\begin{equation}\label{yu-6-21-4}
	I_2\leq 2\epsilon_2^2\left(\int_{D_j}|p_{j+1}|^2\eta_j^2dxd\xi\right)^2
	+\frac{2\pi^4m^4}{R^4\epsilon_2^2}\left(\int_{D_j}|p_j|^2dxd\xi\right)^2, \;\;\forall \ \epsilon_2>0.
\end{equation}
Taking $\epsilon_1=\sqrt{2}/(4\lambda^2)$, $\epsilon_2=1/4$
 in  (\ref{yu-6-21-2})
		and (\ref{yu-6-21-4}), respectively,
we derive that
\begin{eqnarray}\label{yu-6-21-8}
	\sum_{i=1}^2I_i&\leq&\frac{\lambda^{-2}}{4}\left(\int_{D_j}|\nabla p_j|^2\eta_j^2dxd\xi
	\right)^2+\frac{1}{8}\left(\int_{D_j}|p_{j+1}|^2\eta_j^2dxd\xi\right)^2\nonumber\\
	&\;&+\frac{(2+\lambda^{6})4^{2}\pi^4m^{4}}{R^{4}}
	\left(\int_{D_j}|p_j|^2dxd\xi\right)^2.
\end{eqnarray}	
	On the other hand, by the uniform ellipticity condition (\ref{yu-11-28-2}),  we find that
\begin{equation*}\label{yu-6-22-2}
	\left(\int_{D_j}\nabla \bar{p}_j\cdot(A\nabla p_j)\eta_j^2dxd\xi\right)^2
	\geq \lambda^{-2}\left(\int_{D_j}|\nabla p_j|^2\eta_j^2dxd\xi\right)^2.
\end{equation*}
This, together with  (\ref{yu-6-19-9}) and (\ref{yu-6-21-8}), gives that for each $j\in\{0,1,\ldots,m-1\}$,
\begin{eqnarray}\label{yu-6-22-3}
	\int_{D_{j+1}}|p_{j+1}|^2dxd\xi&=&\int_{D_{j+1}}|p_{j+1}|^2\eta^2_jdxd\xi
\leq \sqrt{\frac{4^{3}(2+\lambda^{6})\pi^4m^{4}}{3R^{4}}}
	\int_{D_j}|p_j|^2dxd\xi\nonumber\\
	&\leq&\frac{C_{4}m^2}{R^2}\int_{D_j}|p_j|^2dxd\xi,
\end{eqnarray}
	where $C_{4}=8\pi^2\sqrt{2+\lambda^{6}}$.
	 Here, we used the definition of $D_j$. Iterating (\ref{yu-6-22-3}) for each
	$j\in\{0,1,\ldots,m-1\}$, by the fact that $p_0=p=\hat{v}$ we obtain \eqref{yu-6-22-4}
	and complete the proof.	
\qed

\bigskip

\noindent \textbf{Acknowledgement}. The authors thank the financial support by
the National Natural Science Foundation of China under grants 11971363.

\end{document}